\documentclass{amsart}
\usepackage{graphicx}
\usepackage{amsmath}
\usepackage{color}
\usepackage{amsthm}
\usepackage{hyperref}
\usepackage{bm}
\usepackage{xy}
\usepackage{enumitem}
\usepackage{mathrsfs}
\usepackage[gen]{eurosym}
\usepackage{subcaption}

\theoremstyle{plain}
\newtheorem{theorem}{Theorem}[section]

\newtheorem{corollary}[theorem]{Corollary}

\newtheorem{lemma}[theorem]{Lemma}

\newtheorem{proposition}[theorem]{Proposition}

\newtheorem{definition}[theorem]{Definition}
\newtheorem{condition}[theorem]{Condition}
\newtheorem{assumption}[theorem]{Assumption}
\theoremstyle{remark}
\newtheorem{remark}[theorem]{Remark}
\newtheorem{example}[theorem]{Example}

\numberwithin{equation}{section}

\DeclareMathOperator*{\argmin}{arg\,min}

\begin{document}
	\title[Tsallis and R\'{e}nyi Deformations Linked via a New $\lambda$-Duality]{Tsallis and R\'{e}nyi Deformations\\Linked via a New $\lambda$-Duality}

\begin{abstract}
Tsallis and R\'{e}nyi entropies, which are monotone transformations of each other, are deformations of the celebrated Shannon entropy. Maximization of these deformed entropies, under suitable constraints, leads to the $q$-exponential family which has applications in non-extensive statistical physics, information theory and statistics. In previous information-geometric studies, the $q$-exponential family was analyzed using classical convex duality and Bregman divergence.  In this paper, we show that a generalized $\lambda$-duality, where $\lambda = 1 - q$ is the constant information-geometric curvature, leads to a generalized exponential family which is essentially equivalent to the $q$-exponential family and has deep connections with R\'{e}nyi entropy and optimal transport. Using this generalized convex duality and its associated logarithmic divergence, we show that our $\lambda$-exponential family satisfies properties that parallel and generalize those of the exponential family. Under our framework, the R\'{e}nyi entropy and divergence arise naturally, and we give a new proof of the Tsallis/R\'{e}nyi entropy maximizing property of the $q$-exponential family. We also introduce a $\lambda$-mixture family which may be regarded as the dual of the $\lambda$-exponential family, and connect it with other mixture-type families. Finally, we discuss a duality between the $\lambda$-exponential family and the $\lambda$-logarithmic divergence, and study its statistical consequences.
\end{abstract}
	
	\thanks{}
	\keywords{Tsallis entropy, R\'{e}nyi entropy, exponential family, $q$-exponential family, mixture family, logarithmic divergence, convex duality, information geometry, optimal transport.}

	
	\author{Ting-Kam Leonard Wong}
	\address{Department of Statistical Sciences, University of Toronto}
	\email{tkl.wong@utoronto.ca}
	
	\author{Jun Zhang}
	\address{Department of Psychology and Department of Statistics, University of Michigan}
	\email{junz@umich.edu}
	
	
	\maketitle

	\section{Introduction}
	\subsection{Background} \label{sec:background}
	%
	%
	%
	%
	Exponential families of probability distributions play important roles in probability, statistics, information theory, statistical physics, among other fields \cite{a16,WJ08}. An {\it exponential family} is a parameterized probability density (with respect to a given reference measure $\nu$) of the form
	\begin{equation} \label{eqn:exponential.family}
		p(x; \theta) = e^{\theta \cdot F(x) - \phi(\theta)},
	\end{equation}
	where $F = (F_1, \ldots, F_d)$ is a vector of statistics, $\theta \cdot F = \sum_{i = 1}^d \theta_i F_i$ is the dot product, and $\phi(\theta)$ is the cumulant generating function. In statistical physics, the divisive normalization $Z(\theta) = e^{\phi(\theta)}$ is called the partition function \cite{N11}. Densities of the form \eqref{eqn:exponential.family} maximize the Shannon entropy subject to constraints on the expected value of $F$, thus generalizing the Boltzmann-Gibbs distribution.
	
	Another very useful family of probability distributions is the {\it mixture family}. Let $p_0(x), p_1(x), \ldots, p_d(x)$ be given densities. The mixture family consists of the convex combinations
	\begin{equation} \label{eqn:mixture.family}
		p(x; \eta) = \sum_{i = 0}^d \eta_i p_i(x),
	\end{equation}
	where the mixture parameters satisfy $\eta_i \geq 0$ and $\sum_{i = 0}^d \eta_i = 1$. As reviewed in Section \ref{sec:Legendre.duality}, the exponential and mixture families are closely related as both families are dually flat in the information-geometric sense and can be analyzed via convex duality, Shannon entropy and the Kullback-Leibler (KL) divergence.

	To account for non-exponential -- especially power-law -- behaviours in statistical physics, Tsallis \cite{TC88} introduced a generalized entropy, now called the {\it Tsallis entropy}, which for a given parameter $q \in \mathbb{R} \setminus \{1\}$ (the entropic index) and a given reference measure $\nu$ is defined by
	\begin{equation} \label{eqn:Tsallis.entropy}
		{\bf H}_q^{\mathrm{Tsallis}} (p) = \frac{1}{q - 1} \left( 1 - \int p^q(x) \mathrm{d} \nu(x)  \right),
	\end{equation}
	where $p$ is a probability density with respect to $\nu$. The Tsallis entropy is closely related to the {\it R\'{e}nyi entropy} \cite{R61}, which for $q \in (0, \infty) \setminus \{1\}$ is given by
	\begin{equation} \label{eqn:Renyi.entropy}
		{\bf H}^{\mathrm{R\acute{e}nyi}}_q(p) = \frac{1}{1 - q} \log \left( \int p^q(x) \mathrm{d}  \nu(x)  \right).
	\end{equation}
	When $q > 0$, the Tsallis and R\'{e}nyi entropies are monotonic functions of each other:
	\begin{equation} \label{eqn:Tsallis.Renyi}
		{\bf H}_q^{\mathrm{Tsallis}}(p) = \frac{1}{1 - q} \left(e^{(1 - q) {\bf H}^{\mathrm{R\acute{e}nyi}}_q(p)} - 1 \right).
	\end{equation}
	So maximizing the Tsallis entropy is equivalent to maximizing the R\'{e}nyi entropy. Note that letting $q \rightarrow 1$ in \eqref{eqn:Tsallis.entropy} and \eqref{eqn:Renyi.entropy} recovers the classical {\it Shannon entropy}
	\begin{equation} \label{eqn:Shannon.entropy}
	{\bf H}(p) = -\int p(x) \log p(x) d\nu(x).
\end{equation}
Thus, both the Tsallis and R\'{e}nyi entropies are {\it deformations} of the Shannon entropy. Some arguments that favour the R\'{e}nyi entropy as a physical concept are given in \cite[Section 9.3]{N11}; for example, the R\'{e}nyi entropy satisfies the additive property under independence but the Tsallis entropy does not.
	
	Maximization of Tsallis entropy, subject to constraints on the {\it escort expectation} (see Section \ref{sec:classic.deformation}), leads to a generalized exponential family called the {\it $q$-exponential family}. The idea is to consider a {\it deformation} of the exponential function \cite{N02}. For $q \in \mathbb{R} \setminus \{1\}$, define the {\it $q$-exponential function} $\exp_q : \mathbb{R} \rightarrow [0, \infty]$ by
	\begin{equation} \label{eqn:q.exp}
		\exp_q(t) = \left[1 + (1 - q) t \right]_+^{1/(1 -q )},
	\end{equation}
	where $t_+ = \max\{t, 0\}$ and by convention $0^{t} = \infty$ for $t < 0$. We have $\frac{d}{dt} \exp_q(t) = [ \exp_q(t) ]^q$ when $1 + (1 - q)t > 0$; so $\exp_q$ is convex if and only if $q > 0$.  For this reason, it is natural to restrict $q$ to be strictly positive as in e.g.~\cite{AO11, N11} and later in this paper. The $q$-exponential family, which is an instance of the more general $\phi$-exponential family to be mentioned below, is then defined as the following parameterized density with {\it subtractive} normalization:
	\begin{equation} \label{eqn:q.exponential.family}
		p(x; \theta) = \exp_q \left( \theta \cdot F(x) - \phi_q(\theta) \right).
	\end{equation}
	The normalization $\int p_{\theta} d\nu = 1$ determines the function $\phi_q(\theta)$ which we call the {\it (subtractive) $q$-potential function}. Note that when $q \rightarrow 1$ we recover the exponential family \eqref{eqn:exponential.family}. Contrary to the classical exponential family, the support of the density \eqref{eqn:q.exponential.family} may depend on the parameter $\theta$; a specific example, namely the {\it $q$-Gaussian distribution}, is discussed in Example \ref{ex:q.Guassian}. In the literature the $q$-exponential function is also used to define parameterized densities under {\it divisive} normalization, i.e.,
	\begin{equation} \label{eqn:q.exp.divisive}
    p(x; \vartheta) = \frac{1}{Z_q(\vartheta)} \exp_q(\vartheta \cdot F(x)),
	\end{equation}
	where $\vartheta$ is an alternative parameter; see \cite[Section 4]{N09} for examples. In Propositions \ref{prop:lambda.as.reparamterize} and \ref{prop:q.to.lambda}, we give precise relationships between the parameters $\theta$ and $\vartheta$. Nevertheless, systematic information-geometric studies of the $q$-exponential families as well as other deformed exponential families typically apply convex duality to the {\it subtractive} potential function which can be shown to be convex \cite{AO11, AOM12}. By the {\it Tsallis deformation} we mean the classical framework under which \eqref{eqn:q.exponential.family} is analyzed using  standard convex duality; see Section \ref{sec:classic.deformation}.
		
	Distributions of the form \eqref{eqn:q.exponential.family} have found numerous applications in statistical physics \cite{T09}; a specific example is the momentum distribution of cold atoms in dissipative optical lattices \cite{DBR06}. Recently, the $q$-exponential family has also been applied to statistics and machine learning; see for example \cite{D13,LKLCO19,MTFAFBN21,BMBWGSN20,MNWDR21,NCMQW17}.  In the literature one can find more general formulations of deformed exponential families such as the $\phi$-exponential family \cite{NJ04}, the conjugate $(\rho, \tau)$-embedding \cite{ZJ04,ZJ13,ZJ15} by the second author, and the $U$-model \cite{ES06}. Note that all of these families are studied under subtractive normalization as in \eqref{eqn:q.exponential.family}. In this paper, we focus on the $q$-exponential family which is the simplest deformation to the exponential function yet with many applications. 
	
	On the other hand, motivated by optimal transport \cite{V03} and the duality between Bregman divergence and exponential family \cite{BMDG05}, the first author considered in \cite{W18} a deformed exponential family called the {\it $\mathcal{F}^{(\pm \alpha)}$-families}, where $\alpha > 0$ and 
	\begin{equation} \label{eqn:F.alpha}
		p(x; \vartheta) = \left\{\begin{array}{ll}
			(1 + \alpha \vartheta \cdot F(x))^{\frac{-1}{\alpha}} e^{\varphi(\vartheta)} & \text{($\mathcal{F}^{(\alpha)}$-family)} \\
			(1 + \alpha \vartheta \cdot F(x))^{\frac{1}{\alpha}} e^{-\varphi(\vartheta)} & \text{($\mathcal{F}^{(-\alpha)}$-family)}
		\end{array}\right.
	\end{equation}
    The similarity between this and \eqref{eqn:q.exp.divisive} will be explained below; see in particular \eqref{eqn:lambda.exp.intro}. While the use of divisive normalization (as in \eqref{eqn:q.exp.divisive}) is not new, the novelty of this approach is to analyze the divisive potential $\varphi(\vartheta)$ using a {\it generalized convex duality} motivated by optimal transport. For example, in the $\mathcal{F}^{(\alpha)}$ case it can be shown under suitable conditions that $e^{\alpha \varphi}$ is concave; following \cite{W18}, we say that this $\varphi$ is {\it $\alpha$-exponentially concave}. Under this framework, which adopts a divisive normalization and a generalized duality, the R\'{e}nyi entropy and divergence arise naturally. Hence, we call this approach the {\it R\'{e}nyi deformation}. In this paper, we present a novel framework, consisting of the $\lambda$-duality and the $\lambda$-exponential family, that unifies the $\mathcal{F}^{(\pm \alpha)}$-families and links it with the Tsallis deformation.

	\subsection{Outline}
	The main idea of this paper is a {\it $\lambda$-duality} which is a deformation of the usual convex duality reviewed in Section \ref{sec:Legendre.duality}. The $\lambda$ parameter is related to the classical $q$ parameter by $\lambda = 1 - q$ and this is always assumed in the paper. In a nutshell, instead of convex functions, for $\lambda \neq 0$ fixed we work with functions $f$, defined on a given open convex set of $\mathbb{R}^d$, such that $\frac{1}{\lambda} (e^{\lambda f} - 1)$ is convex. These functions are related to exponentially convex and exponentially concave functions in the literature \cite{PW18, W18} (see Remark \ref{rmk:exp.concavity.convexity}). Also, instead of the convex conjugate we use the {\it $\lambda$-conjugate} given by
	\begin{equation} \label{eqn:c.lambda.convexity.intro}
		f^{c_{\lambda}}(v) = \sup_u \left\{ -c_{\lambda}(u, v) - f(u) \right\} = \sup_u \left\{ \frac{1}{\lambda} \log (1 + \lambda u \cdot v) - f(u) \right\},
	\end{equation}
	where $c_{\lambda}(u, v) = -\frac{1}{\lambda} \log (1 + \lambda u \cdot v)$ is the cost function in the sense of optimal transport. It deforms the usual convex duality and, as we will see below, is naturally compatible with R\'{e}nyi entropy and divergence. The $\lambda$-duality also leads to the following generalization of the Bregman divergence:
	\begin{equation*}
		{\bf L}_{\lambda, f}[u : u'] 
		= f(u) - f(u') - \frac{1}{\lambda} \log (1 + \lambda \nabla f(u') \cdot (u - u')).
	\end{equation*}
	We call this the {\it $\lambda$-logarithmic divergence}. The details of this duality, which was motivated by optimal transport and previous works of Pal and the first author \cite{PW16, PW18, PW18b, W18, W19, WY19b}, are given in Section \ref{sec:generalized.duality}.

	Using the $\lambda$-duality we study the $q$-exponential family from a new perspective. In Section \ref{sec:lambda.exp.family} we introduce the {\it $\lambda$-exponential family} (Definition \ref{def:lambda.exp.family}) which is essentially the $q$-exponential family under divisive normalization (see \eqref{eqn:q.exp.divisive}):
	\begin{equation} \label{eqn:lambda.exp.intro}
		p(x; \vartheta) = e^{-c_{\lambda}(\vartheta, F(x)) - \varphi_{\lambda}(\vartheta)}\\
		= \exp_{q}(\vartheta \cdot F(x)) e^{-\varphi_{\lambda}(\vartheta)},
	\end{equation}
	where $q = 1 - \lambda$ and $\vartheta$ is another natural parameter related to $\theta$ via $\theta = \vartheta e^{-\lambda \varphi_{\lambda}(\vartheta)}$. The precise relationships between \eqref{eqn:lambda.exp.intro} and \eqref{eqn:q.exponential.family} are given in Propositions \ref{prop:lambda.as.reparamterize} and \ref{prop:q.to.lambda}. This $\lambda$-exponential family unifies the $\mathcal{F}^{(\pm \alpha)}$-families \eqref{eqn:F.alpha}. For an exponential family ($\lambda \rightarrow 0$), the subtractive and divisive normalizations are the same because the exponential function satisfies the functional equation $\exp(s + t) = \exp(s) \exp(t)$. We show that the divisive representation \eqref{eqn:lambda.exp.intro} of the $q$-exponential family is naturally compatible with the $\lambda$-duality, in the same way that convex duality describes the pairing of the log-partition function with the negative Shannon entropy for the exponential family. In particular, when $q = 1 - \lambda > 0$ and under suitable regularity conditions, the {\it divisive $\lambda$-potential} $\varphi_{\lambda}$ defined by \eqref{eqn:lambda.exp.intro} is $c_{\lambda}$-convex and its $\lambda$-logarithmic divergence is the R\'{e}nyi divergence of order $q$. Using the $\lambda$-duality, we give a new proof of the R\'{e}nyi entropy maxmizing property under constraints on the escort expectation.
	
	Note that in Section \ref{sec:background} we have not mentioned the ``$q$-analogue" of the mixture family. While generalized mixture families have been considered in the literature (see \cite[Section 4.2]{a16}), they are not usually studied together with deformed exponential families under a unified framework. In Section \ref{sec:mixtures} we study mixture-type families under the $\lambda$-duality. We first show that a $\lambda$-exponential family is closed under the $\alpha$-mixture of Amari \cite{A07} (also see Example \ref{eg:alpha.family}). Then we introduce a new {\it $\lambda$-mixture family} (Definition \ref{def:lambda.mixture}) which is in some sense dual to the $\lambda$-exponential family. 
	
	In Section \ref{sec:information.geometry} we describe the information geometry of  $\lambda$-exponential and $\lambda$-mixture families induced by the associated $\lambda$-logarithmic divergences or equivalently the R\'{e}nyi divergences. Inheriting the results from \cite{PW16,W18}, this geometry is dually projectively flat with constant sectional curvature $\lambda$, and the divergence satisfies a generalized Pythagorean theorem. Section \ref{sec:dual.divergence} explores further the relationship between $\lambda$-exponential family and $\lambda$-logarithmic divergence, and discusses some statistical implications. Finally, in Section \ref{sec:conclusion} we conclude and point out several directions for further research.
	
	\section{$\lambda$-duality as a deformation of convex duality} \label{sec:generalized.duality}
	\subsection{Review of convex duality} \label{sec:Legendre.duality}
	The key mathematical concept which underlies the exponential family, as well as previous treatments of the $q$-exponential family (and other deformed exponential families under subtractive normalization, as in e.g.~\cite{AOM12}), is the Legendre duality of convex functions. To wit, the cumulant generating function $\phi(\theta)$ in \eqref{eqn:exponential.family}, and the $q$-potential function $\phi_q(\theta)$ in \eqref{eqn:q.exponential.family} for $q > 0$, are convex functions of the parameter $\theta$ \cite[Theorem 2]{AO11}. 
	
	For readability we do not spell out all technical conditions in this brief review, and refer the reader to \cite{B14,R70} for a comprehensive treatment of convex analysis on Euclidean space and its application to exponential family. Given a function $f$ on $\mathbb{R}^d$, its convex conjugate is defined by
	\begin{equation} \label{eqn:convex.conjugate}
		f^*(v) = \sup_u \left( u \cdot v - f(u) \right), \quad v \in \mathbb{R}^d.
	\end{equation}
	Then $f$ is convex and lower-semicontinuous if and only if $f^{**} = f$. When $f$ is strictly convex and differentiable, the Legendre transformation
	\begin{equation} \label{eqn:convex.gradient}
		v = \nabla f(u),
	\end{equation}
	which can be motivated by the first order condition in \eqref{eqn:convex.conjugate}, defines a ``dual coordinate'' $v$, and its inverse is given by $v = \nabla f^*(u)$. Brenier's theorem in optimal transport theory \cite{V03} states that the Legendre transformation is an optimal transport map under the quadratic cost $c(u, v) = \frac{1}{2}|u - v|^2$. The convex function $f$ also induces a {\it Bregman divergence}, which is widely applied in statistics and machine learning, by
	\begin{equation} \label{eqn:Bregman.divergence}
		{\bf B}_{f}[u : u'] = f(u) - f(u') - \nabla f(u') \cdot (u - u') \geq 0.
	\end{equation}
	
	Now consider the cumulant generating function $\phi(\theta)$ of an exponential family \eqref{eqn:exponential.family}, where $\theta$, the primal variable, is the natural parameter. Then $\phi$ is convex and the dual variable $\eta = \nabla \phi(\theta)$, under the Legendre duality, is the {\it expectation parameter} given by
	\begin{equation} \label{eqn:expectation.parameter}
		\eta = \mathbb{E}_{\theta} [ F(X) ]  = \int F(x) p(x;\theta) \mathrm{d}\nu(x),
	\end{equation}
	where under $\mathbb{E}_{\theta}$ the random variable $X$ is distributed according to the density $p(\cdot ; \theta)$. As a function of $\eta$, the Legendre conjugate $\psi = \phi^*$ is the negative Shannon entropy, namely $\psi(\eta) = -{\bf H}(p_{\theta}) = - \int p_\theta \log p_\theta  \mathrm{d} \nu$, where for notational simplicity we write $p_{\theta} = p(\cdot; \theta)$. Furthermore, the Bregman divergences of $\phi$ and $\psi$ can be expressed as {\it Kullback-Leilber (KL) divergences}: 
	\begin{equation} \label{eqn:exp.family.Bregman}
		{\bf B}_{\phi}[\theta : \theta'] = {\bf B}_{\psi}[\eta' : \eta] =  {\bf H}( p_{\theta'}  || p_{\theta} ),
	\end{equation}
where
\begin{equation} \label{eqn:KL.divergence}
{\bf H}( p  || p' ) = \int  p \log \frac{ p }{ p' } \mathrm{d}\nu.
\end{equation}
Consequently, the local second order approximation of ${\bf B}_{\phi}$, which defines a Riemannian metric on the exponential family regarded as a statistical manifold, is given by the {\it Fisher information metric}. Explicitly, we have
	\begin{equation} \label{eqn:Breman.div.quadratic.approx}
		{\bf B}_{\phi}[ \theta + \Delta \theta : \theta] = \frac{1}{2} (\Delta \theta)^{\top} g(\theta) (\Delta \theta) + O(|\Delta \theta|^3),
	\end{equation}
	where
	\begin{equation} \label{eqn:Fisher.metric}
		g_{ij}(\theta) = \int \frac{\partial \log p_{\theta}}{\partial \theta_i} \frac{\partial \log p_{\theta}}{\partial \theta_j} p_{\theta} \mathrm{d}\nu
	\end{equation}
	defines the Fisher information metric. Note that the same metric is obtained if we expand ${\bf B}_{\phi}[\theta : \theta + \Delta \theta]$ instead. The natural parameter $\theta$ and the expectation parameter $\eta$ can be regarded as two sets of affine coordinate systems that are ``dual'' with respect to the metric $g$. The corresponding {\it dually flat geometry} is well-studied in information geometry \cite{a16}.
	
Similarly, for a mixture family \eqref{eqn:mixture.family}, it can be shown that the negative Shannon entropy $\psi(\eta) = -{\bf H}(p_{\eta})$, where $p_{\eta} = p(\cdot ; \eta)$, is a convex function of the mixture parameter, and its Bregman divergence is again a KL-divergence:
	\begin{equation} \label{eqn:mixture.family.Bregman}
		{\bf B}_{\psi}[\eta : \eta'] =  {\bf H}( p_{\eta} || p_{\eta'}).
	\end{equation}
	The induced Riemannian metric is again the Fisher metric. Thus, convex duality and Bregman divergence underlie both the exponential and mixture families. 
	
	\subsection{Classical deformation theory} \label{sec:classic.deformation}
	\begin{figure}[!t]
		\centering
		\includegraphics[width=3in]{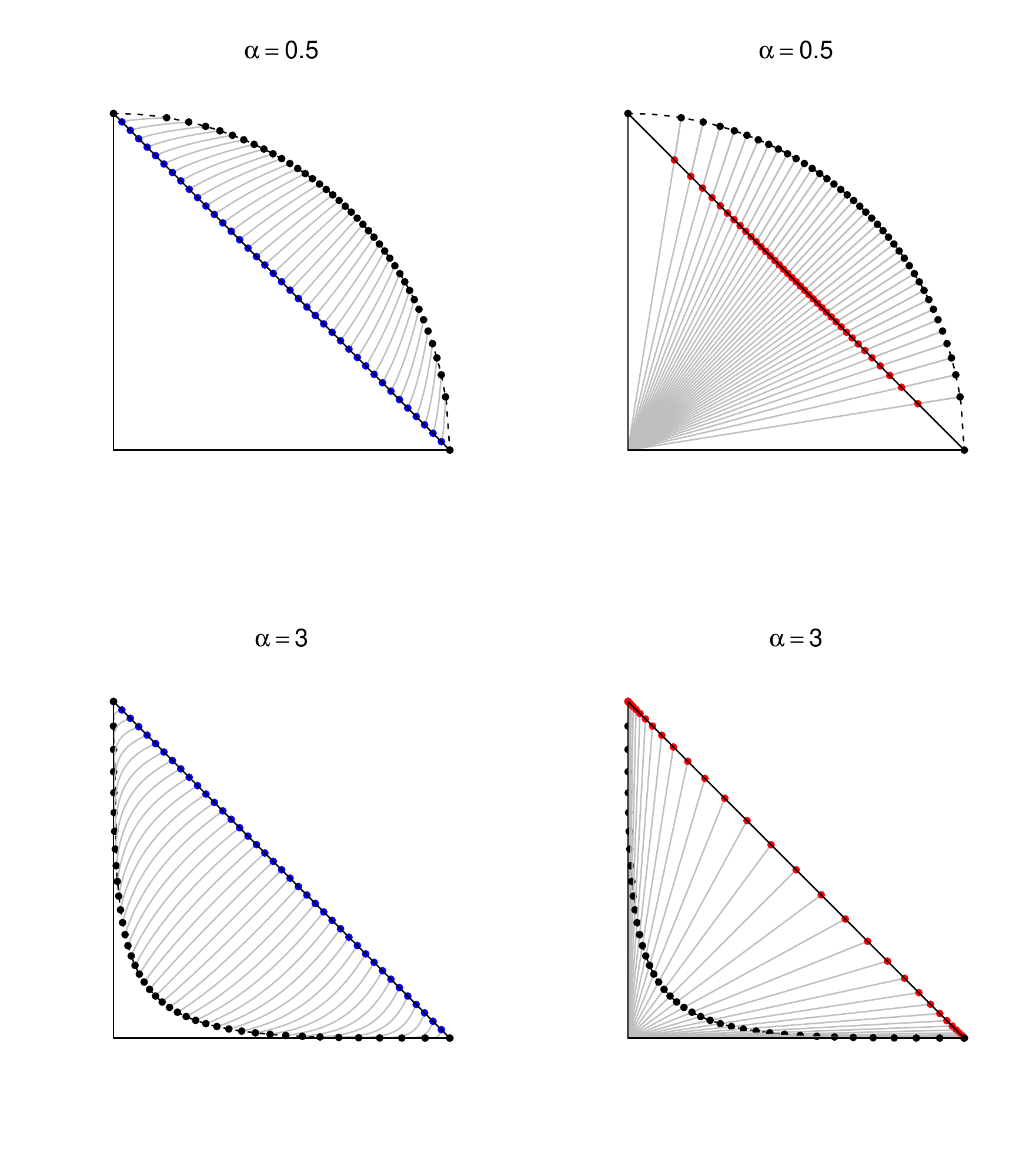}
		\vspace{-0.8cm}
		\caption{Illustration of the escort transformation $p \mapsto \mathcal{E}_{\alpha}[p]$, where $p = (p_1, p_2)$ is a probability vector with length $2$. Left: $p$ (in blue) on the unit simplex is first mapped to $p^{\alpha} = (p_1^{\alpha}, p_2^{\alpha})$ (in black). The curve (in grey) shows the trajectory $t \mapsto p^{(1 - t) + t\alpha}$ for $0 \leq t \leq 1$. Right: Normalize $p^{\alpha}$ along the straight line (in grey) passing through the origin to obtain the escort distribution $\widetilde{p} = \mathcal{E}_{\alpha}[p] = p^{\alpha}/\int p^{\alpha} d\nu$ (in red). (Top row: $\alpha = 0.5$. Bottom row: $\alpha = 3$.) Note that when $\alpha < 0$ the image of $p \mapsto p^{\alpha}$ becomes unbounded.}
		\label{fig:escort}
	\end{figure}
	
	Consider now a $q$-exponential family \eqref{eqn:q.exponential.family} which is a deformation of the exponential family. Here, the exponential function $\exp$ is replaced by the deformed exponential $\exp_q$ given by \eqref{eqn:q.exp}. When $q > 0$, the $q$-potential function $\phi_q$ in \eqref{eqn:q.exponential.family} can be shown to be convex (this requires differentiability under the integral sign), and hence defines a dual variable $\eta = \nabla \phi_q(\theta)$ via the Legendre transformation. To interpret $\eta$ probabilistically we recall the concept of {\it escort distribution} which arises naturally in the study of generalized exponential families and their applications in non-extensive statistical physics \cite{N11}.  Given a probability density $p(x)$ (or more generally a non-negative function which is not $\nu$-almost everywhere zero) with respect to the reference measure $\nu$, and an exponent $\alpha \in \mathbb{R} \setminus \{0\}$, we define the {\it escort distribution with exponent $\alpha$} be the density
	\begin{equation} \label{eqn:escort.transformation.def}
		\widetilde{p} = \mathcal{E}_{\alpha}[p] := \frac{p^{\alpha}}{\int p^{\alpha} \mathrm{d}\nu},
	\end{equation}
	provided that the integral $\int p^{\alpha} \mathrm{d}\nu$ is finite. Some physical and information-theoretic interpretations of the escort distribution are given in \cite{Abe03, B09}. Now we may interpret the dual parameter $\eta$ as an {\it escort expectation}:
	\begin{equation} \label{eqn:escort}
		\eta = \widetilde{\mathbb{E}}_{\theta} [ F(X) ] := \int F(x) \widetilde{p}_{\theta}(x) \mathrm{d}\nu(x),
	\end{equation}
	where $\widetilde{p}_{\theta} = \mathcal{E}_{q}[p_{\theta}]$ is the escort distribution with exponent $q$. It is helpful to think of the escort transformation as a composition of two operations, namely $p \mapsto p^{\alpha}$ and the normalization $p^{\alpha} \mapsto p^{\alpha} / \int p^{\alpha} \mathrm{d}\nu$. In Figure \ref{fig:escort} we illustrate these operations where $p$ is a density function on a two-point set with respect to the counting measure, so $p$ can be identified with a probability vector $p = (p_1, p_2)$. For an exponential family the escort transformation is equivalent to a dilation with respect to the natural parameter, i.e., $\mathcal{E}_{\alpha}[p_{\theta}] = p_{\alpha \theta}$, whenever $\alpha \theta$ belongs to the parameter set. This corresponds to the fact that the escort transformation is the scalar multiplication under the {\it Aitchison geometry} in compositional data analysis \cite{PB11}.
	
	It can be shown (see e.g.~\cite[Section 4]{AO11}) that densities of the form \eqref{eqn:q.exponential.family} maximize the Tsallis entropy under constraints on the   escort expectation; in Theorem \ref{thm:Renyi.entropy.maximize}, we give a new proof of this result using our $\lambda$-duality. The Tsallis (or equivalently R\'{e}nyi) entropy maximization property gives a theoretical justification of the $q$-exponential family. Nevertheless, other fundamental properties of exponential families described in Section \ref{sec:Legendre.duality} do not have ``exact'' analogues in previous treatments of the $q$-exponential family. For example, as first shown in \cite{AO11}, the Bregman divergence of $\phi_q$ is not the Tsallis relative entropy, and the associated Riemannian metric is not the Fisher metric but is a conformal transformation of it. Also see \cite{SMW20} for a more recent attempt. In this paper, we show that our framework using $\lambda$-duality provides natural and elegant statements which nicely parallel the case of exponential and mixture families.

	\subsection{Generalized $\lambda$-duality} \label{sec:lambda.duality}
	The key idea of this paper is the following. Instead of deforming the exponential and logarithm functions, we deform the notion of conjugation, and hence the convex duality, by replacing the pairing $u \cdot v$ in \eqref{eqn:convex.conjugate} by a nonlinear function $-c_{\lambda}$ of it. This generalized duality, which can be defined for a generic cost function $c(u, v)$, is well-known in the optimal transport literature \cite{S15,V03,V08} in characterizations of optimal transport plans. The classical Legendre duality corresponds to $c(u, v) = -u \cdot v$ which is the cross term of the quadratic cost $\frac{1}{2}|u - v|^2$ when expanded. In what follows we let a constant $\lambda \in \mathbb{R} \setminus \{0\}$ be given and, for reasons that will become clear in Section \ref{sec:information.geometry}, call it the {\it curvature parameter}. It is related to the $q$ parameter via $\lambda = 1 - q$. Thus the usual exponential family (the limit as $\lambda \rightarrow 0$) corresponds to zero curvature (dual flatness). When applying the $\lambda$-duality to the $\lambda$-exponential family, we typically assume (as in Section \ref{sec:classic.deformation}) that $q > 0$ (so that $\exp_q$ is convex), or equivalently $\lambda < 1$. The specific functional form of the logarithmic cost function $c_{\lambda}$ is motivated by previous works of Pal and the first author \cite{PW16,PW18,PW18b,W18} which led to tractable results; see in particular \cite{PW18b} where we obtained an analogue of Brenier's theorem for the Dirichlet transport problem on the unit simplex. Here we streamline the treatment by introducing the unifying parameter $\lambda$ (rather than $\pm \alpha$ as in \eqref{eqn:F.alpha}), and work instead with $c$-convex functions (rather than $c$-concave functions which are more common in optimal transport theory) so that the notations parallel those in Section \ref{sec:Legendre.duality}. 
	
	\medskip
	
	By continuity, we let $\log t = -\infty$ for $t \leq 0$ and $e^{-\infty} = 0$. In particular, we have
	\begin{equation} \label{eqn:e.log}
		e^{\log t} = t_+, \quad t \in \mathbb{R}.
	\end{equation}
	
	\begin{definition} [$\lambda$-duality]  \label{def:lambda.duality}
		Fix $\lambda \in \mathbb{R} \setminus \{0\}$.
		\begin{itemize}
			\item[(i)] We define $c_{\lambda} : \mathbb{R}^d \times \mathbb{R}^d \rightarrow \bar{\mathbb{R}} := \mathbb{R} \cup \{\pm \infty\}$ by
			\begin{equation} \label{eqn:lambda.cost}
				c_{\lambda}(u, v) = \frac{-1}{\lambda} \log (1 + \lambda u \cdot v).
			\end{equation}
			\item[(ii)] If $f: \Omega \subset \mathbb{R}^d \rightarrow \bar{\mathbb{R}}$, we define  $f^{c_{\lambda}} : \mathbb{R}^d \rightarrow \bar{\mathbb{R}}$ by
			\begin{equation} \label{eqn:lambda.transform}
				f^{c_{\lambda}}(v) = \sup_{u \in \Omega} \left\{ -c_{\lambda}(u, v) - f(u) \right\}, \quad v \in \mathbb{R}^d,
			\end{equation} 
			where by convention we set $+\infty - (+\infty) = -\infty - (-\infty) = -\infty$. We call $f^{c_{\lambda}}$ the $c_{\lambda}$-conjugate of $f$.
			\item[(iii)] A function $f: \Omega \rightarrow \mathbb{R} \cup \{+\infty\}$ is $c_{\lambda}$-convex on $\Omega$ if $f \not\equiv +\infty$ and $f = g^{c_{\lambda}}$ on $\Omega$  for some $\Omega' \subset \mathbb{R}^d$ and $g: \Omega' \rightarrow \mathbb{R} \cup \{+ \infty\}$. 
		\end{itemize}
	\end{definition}
	
	\begin{figure}[!t]
		\centering
		\includegraphics[width=3.48in]{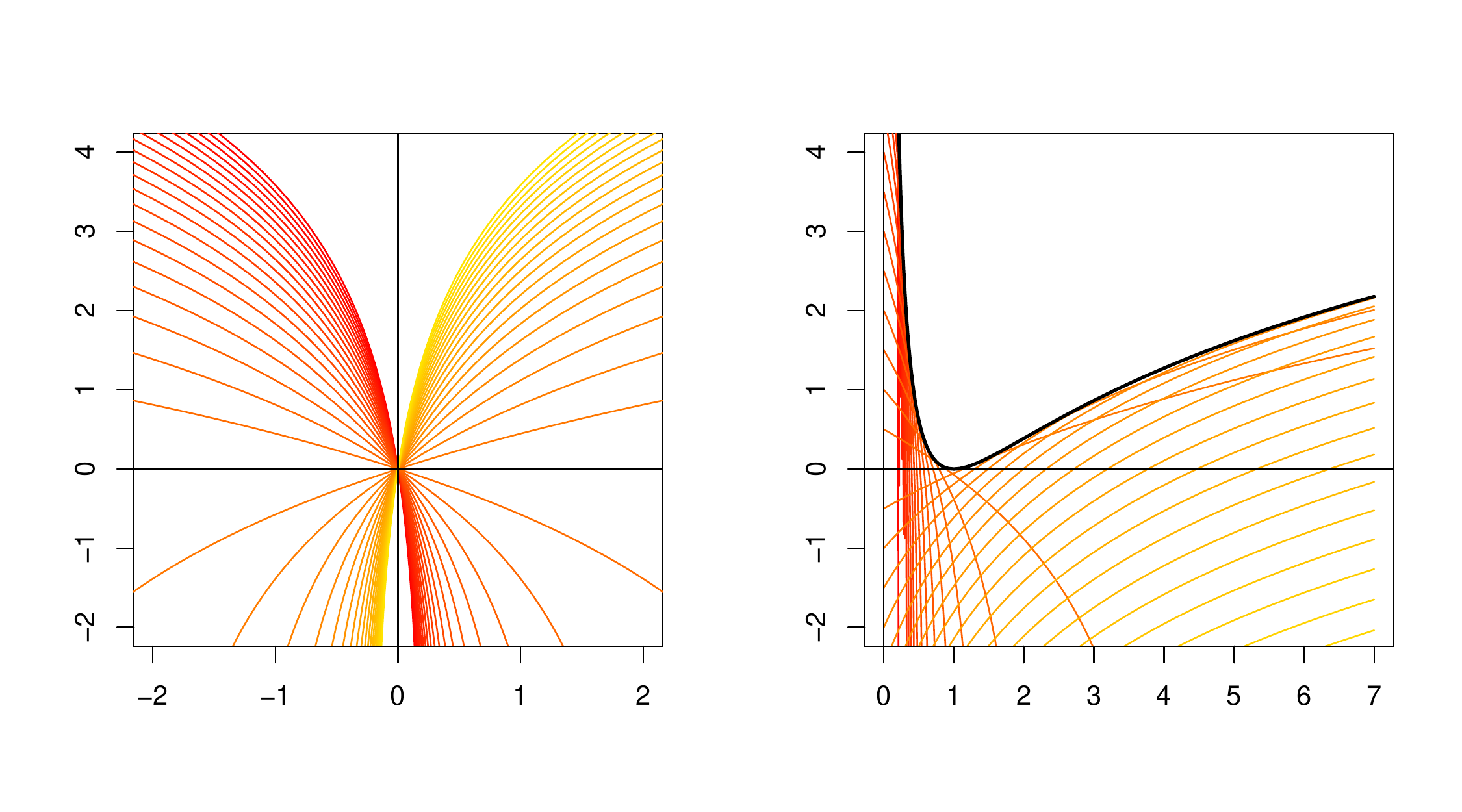}
		\vspace{-0.7cm}
		\caption{Illustration of $\lambda$-duality on the real line where $\lambda = 0.5$. Left: Graphs of the mappings $u \mapsto \frac{1}{\lambda} \log(1 + \lambda uv)$ for several values of $v$ in the interval $[-5, 5]$. Right: The function $f(u) = \frac{1}{\lambda} \left( \frac{1}{u} - 1 + \log u\right)$, which is $c_{\lambda}$-convex on $\Omega = (0, \infty)$, shown as the upper envelope of functions of the form $\frac{1}{\lambda} \log(1 + \lambda uv) - g(v)$, where $g(v) = v$. Note that $f$ is not convex in the usual sense. }
		\label{fig:conjugate}
	\end{figure}
	
	
	This generalized duality is illustrated in Figure \ref{fig:conjugate} where $\lambda = 0.5$. Note that for $u, v \in \mathbb{R}^d$ we have $\lim_{\lambda \rightarrow 0} c_{\lambda}(u, v) = u \cdot v$. Thus, when $\lambda \rightarrow 0$ we recover the usual convex duality. Observe also that
	\begin{equation} \label{eqn:exp.c}
		e^{-c_{\lambda}(u, v)} = (1 + \lambda u \cdot v)_+^{1/\lambda} = \exp_{q}  (u \cdot v),
	\end{equation}
	where $q = 1 - \lambda$. As will be seen in Section \ref{sec:lambda.exp.family}, this and the $\lambda$-duality allow us to give an alternative treatment of the $q$-exponential family without deforming the exponential function. 
	
	
	In optimal transport (see e.g.~\cite{AG13}), the operation \eqref{eqn:lambda.transform} for a generic cost function $c(x, y)$ is called the $c_-$-transform (the $c_+$-transform, which involves an infimum rather than a supremum, leads to $c$-concave functions). In general, it is not possible to characterize $c$-concave and $c$-convex functions explicitly in terms of familiar convexity concepts. The cost function $c_{\lambda}$ is quite special as shown by the following result which specifies a class of ``nice'' $c_{\lambda}$-functions useful for our applications. Note that in Definition \ref{def:lambda.duality}(iii) $f$ and $g$ are defined on respective domains $\Omega, \Omega' \subset \mathbb{R}^d$. In the context of Theorem \ref{thm:lambda.duality} below, this allows us to deduce that $1 + \lambda u \cdot v > 0$ for $(u, v) \in \Omega \times \Omega'$ and avoid infinity values. Developing the duality and differential theory in full generality along the lines of Rockafellar's classic treatise \cite{R70} is of independent mathematical interest and is left for future research. Also see the discussion in Section \ref{sec:dual.divergence}.
	
	\begin{theorem} [Main results of $\lambda$-duality] \label{thm:lambda.duality}
		Let $\lambda \neq 0$ and let $\Omega \subset \mathbb{R}^d$ be an open convex set. Consider a smooth function $f: \Omega \rightarrow \mathbb{R}$ such that the Hessian of $F_{\lambda} = \frac{1}{\lambda} (e^{\lambda f} - 1)$ is strictly positive definite (hence $F_{\lambda}$ is strictly convex) and $1 - \lambda \nabla f(u) \cdot u > 0$ on $\Omega$. Then:
		\begin{enumerate}
			\item[(i)] $f$ is $c_{\lambda}$-convex function on $\Omega$.
			\item[(ii)] Define for $u \in \Omega$ the mapping
			\begin{equation} \label{eqn:c.Legendre.transform}
				v = \nabla^{c_{\lambda}} f(u) := \frac{1}{1 - \lambda \nabla f(u) \cdot u} \nabla f(u).
			\end{equation}
			Then $\nabla^{c_{\lambda}} f$ is a diffeomorphism from $\Omega$ onto its range $\Omega'$ which is an open set. We call $\nabla^{c_{\lambda}} f$ the $\lambda$-deformed Legendre transformation, or simply the $\lambda$-gradient.
			\item[(iii)] Consider $f^{c_{\lambda}}$ as a function on the range $\Omega' = \nabla^{c_{\lambda}} f(\Omega)$. Then $(f^{c_{\lambda}})^{c_{\lambda}} = f$ on $\Omega$.
			\item[(iv)] For $u \in \Omega$ and $v = \nabla^{c_{\lambda}} f(u) \in \Omega'$ we have $1 + \lambda u \cdot v > 0$ and the following identity holds:
			\begin{equation} \label{eqn:lambda.Fenchel.conjugate}
				f(u) + f^{c_{\lambda}}(v) \equiv -c_{\lambda}(u, v) = \frac{1}{\lambda} \log (1 + \lambda u \cdot v).
			\end{equation}
			In particular, $f^{c_{\lambda}}$ is smooth on $\Omega'$.
			\item[(v)] The inverse of $\nabla^{c_{\lambda}} f$ is given by
			\[
			\nabla^{c_{\lambda}} f^{c_{\lambda}}(v) := \frac{1}{1 - \lambda \nabla f^{c_{\lambda}}(v) \cdot v} \nabla f^{c_{\lambda}}(v),
			\]
			which is well-defined on $\Omega'$.
		\end{enumerate}
	\end{theorem}
	\begin{proof}
		See \cite[Section 3.3]{W18}. Here we rephrased the results in terms of the $\lambda$-duality. To illustrate some of the ideas involved, we provide the proof of (i) in the Appendix.
	\end{proof}

	The definition of the $\lambda$-gradient $\nabla^{c_{\lambda}} f(u)$ can be motivated by the optimality condition in \eqref{eqn:lambda.transform} (compare with \eqref{eqn:convex.conjugate}). The $\lambda$-gradient (analogous to the Brenier map) can be interpreted as an optimal transport map under the logarithmic cost $c_{\lambda}$. For the geometric meaning of the condition $1 - \lambda \nabla f(u) \cdot u > 0$ see the proof of Theorem \ref{thm:lambda.duality}(i) in the Appendix. Analytically, it allows us to apply convex/concave duality to $e^{\lambda f}$ and then take logarithm to obtain a generalized convex duality based on the logarithmic cost $c_{\lambda}$. This is essentially a normalization which makes $0$ a reference point (see the left panel of Figure \ref{fig:conjugate}). By Theorem \ref{thm:c.exp.potential}, it holds for the divisive $\lambda$-potential of the $\lambda$-exponential family under suitable regularity conditions. Geometrically, $f = \frac{1}{\lambda} \log (1 + \lambda F_{\lambda})$ is simply a multiple of the logarithm of a positive convex/concave function (depending on the sign of $\lambda$). By Theorem \ref{thm:lambda.duality}(ii), it is given on $\Omega$ as the supremum of a collection of vertically translated logarithmic functions (right panel of Figure \ref{fig:conjugate}).

	It is convenient to introduce a terminology for the functions that satisfy the hypotheses of Theorem \ref{thm:lambda.duality}.
	
	\begin{definition} [Regular $c_{\lambda}$-convex function]
		By a regular $c_{\lambda}$-convex function we mean a function $f$ which satisfies the hypotheses of Theorem \ref{thm:lambda.duality}.
	\end{definition}
	
In the context of Theorem \ref{thm:lambda.duality}, the usual convexity is replaced by convexity of the transformation $F_{\lambda} = \frac{1}{\lambda} (e^{\lambda f} - 1)$, and the $\lambda$-gradient $\nabla^{c_{\lambda}} f$ defines a new dual variable. Note that the additive- term $-1/\lambda$ in $F_{\lambda}$ is not necessary and is included so that as $\lambda \rightarrow 0$ we have $F_{\lambda} \rightarrow f$. 
	
If $f$ is $C^2$ (twice continuously differentiable), then $F_{\lambda}$ is convex if and only if the matrix
\begin{equation} \label{eqn:second.derivative.condition}
e^{-\lambda f(u)} \nabla^2 F_{\lambda}(u) = \nabla^2 f(u) + \lambda (\nabla f(u))(\nabla f(u))^{\top}
\end{equation}
is positive semidefinite on $\Omega$ (here and throughout $\nabla f(u)$ is regarded as a column vector and $\cdot^{\top}$ denotes transposition). Note that when $\lambda < 0$ then
\[
\nabla^2 f(u) \succeq (-\lambda)  (\nabla f(u))(\nabla f(u))^{\top} \succeq 0,
\]
so that $f$ itself is convex (here $\succeq$ is the Loewner order). In Section \ref{sec:information.geometry}, we will use \eqref{eqn:second.derivative.condition} to define a Riemannian metric on $\Omega$.
	
	
	\begin{remark} [Exponential convexity and concavity] \label{rmk:exp.concavity.convexity}
		Let $f$ be regular $c_{\lambda}$-convex. If $\lambda > 0$, then $e^{\lambda f}$ is convex on $\Omega$; following \cite{PW16, PW18, W18}, we say that $f$ is {\it $\alpha$-exponentially convex} on $\Omega$ with $\alpha = \lambda$. If $\lambda < 0$, then $e^{|\lambda| (-f)}$ is a positive concave function, and we say that $-f$ is {\it $\alpha$-exponentially concave} with $\alpha = |\lambda|$. The present framework unifies the two cases in \cite[Section 3]{W18}.
	\end{remark}

    We close this subsection by making the observation (not used in the rest of the paper) that the transformation $t \mapsto \frac{1}{\lambda} \log (1 + \lambda t)$, which characterizes the deformation in $c_{\lambda}$, as well as its inverse $s \mapsto \frac{1}{\lambda} (e^{\lambda s} - 1)$, which defines $F_{\lambda}$ in Theorem \ref{thm:lambda.duality}, are closely related to the {\it Box–Cox power transformation} $s^{(\lambda)} = \frac{1}{\lambda} (s^{\lambda} - 1)$ (and its inverse) introduced in \cite{BC64}, where the same $\lambda$ parameter is used. Specifically, we have $F_{\lambda} = (\exp f)^{(\lambda)}$. 
	
	
	\subsection{$\lambda$-logarithmic divergence} \label{sec:lambda.log.divergence}
	Recall the derivation of the Bregman divergence \eqref{eqn:Bregman.divergence}. If $f$ is convex and differentiable, for any $u, u'$ we have
	\begin{equation} \label{eqn:f.convexity}
		f(u') + \nabla f(u') \cdot (u - u') \leq f(u).
	\end{equation}
	The Bregman divergence ${\bf B}_f[u : u']$ is defined by taking the difference.
	
	\begin{figure}[!t]
		\centering
		\includegraphics[width=3.48in]{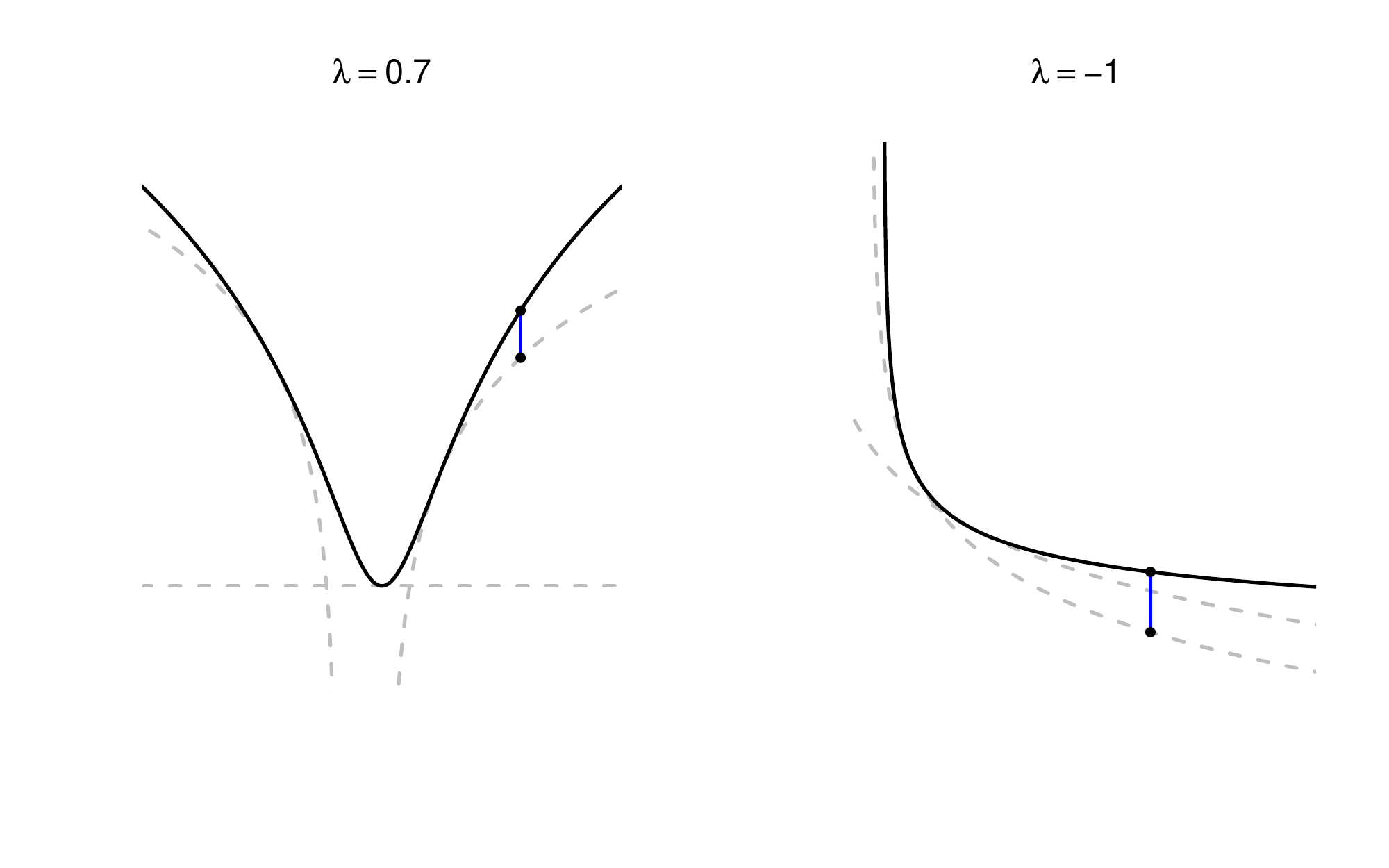}
		\vspace{-1cm}
		\caption{Illustration of $\lambda$-logarithmic divergence where $f$ is defined on an interval of the real line. Left: $\lambda = 0.7$. Right: $\lambda = -1$. The graph of $f$ is shown as a solid curve (in black). The logarithmic first order approximation, which by construction supports the graph of $f$ from below, is shown as a dashed curve (in grey). Its error, shown as the vertical line segment in blue, gives the value of ${\bf L}_{\lambda, f}[u : u']$.}
		\label{fig:ldiv}
	\end{figure} 
	
	Our $\lambda$-duality leads to a different divergence. Consider a function $f$ on a convex set $\Omega$ such that $F_{\lambda} = \frac{1}{\lambda} (e^{\lambda f} - 1)$ is convex. If $f$ is differentiable at $u' \in \Omega$, a convexity argument (write \eqref{eqn:f.convexity} for $F_{\lambda}$ then rearrange) gives the inequality
	\begin{equation} \label{eqn:lambda.log.derivation}
		\frac{1}{\lambda} + \nabla f(u') \cdot (u - u') \leq \frac{1}{\lambda} e^{\lambda( f(u) - f(u'))}, \quad u \in \Omega.
	\end{equation}
	We can define a divergence by taking logarithm on both sides and rearranging. There are two cases depending on the sign of $\lambda$, but the resulting expression is the same. The following definition unifies the $L^{(\pm \alpha)}$-divergences introduced in \cite[Section 3]{W18}. A graphical illustration is shown in Figure \ref{fig:ldiv}. Also see \cite{WY19b} for the general framework of {\it $c$-divergence} which defines divergences based on optimal transport maps.
	
	\begin{definition} [$\lambda$-logarithmic divergence] \label{def:L.div}
		Let $\Omega \subset \mathbb{R}^d$ be convex and let $f : \Omega \rightarrow \mathbb{R}$ be a function such that $\frac{1}{\lambda} (e^{\lambda f} - 1)$ is convex. If $u, u' \in \Omega$ and $f$ is differentiable at $u'$, we define the $\lambda$-logarithmic divergence by
		\begin{equation} \label{eqn:lambda.logdivergence}
			{\bf L}_{\lambda, f}[u : u'] 
			= f(u) - f(u') - \frac{1}{\lambda} \log (1 + \lambda \nabla f(u') \cdot (u - u')).
		\end{equation}
	\end{definition}
	
	From \eqref{eqn:lambda.log.derivation}, we have ${\bf L}_{\lambda, f}[u : u'] \geq 0$; also ${\bf L}_{\lambda, f}[u' : u'] = 0$. When $\lambda > 0$, it is possible that $1 + \lambda \nabla f(u') \cdot (u - u') \leq 0$. When this happens, the definition implies that ${\bf L}_{\lambda, f}[u : u'] = \infty$. On the other hand, when $\lambda < 0$ we have ${\bf L}_{\lambda, f}[u : u'] < \infty$ for all $u, u' \in \Omega$. Clearly, if $\frac{1}{\lambda} (e^{\lambda f} - 1)$ is strictly convex, then ${\bf L}_{\lambda, f}[u : u']$ is strictly positive for $u \neq u'$. For later use, observe that if for a given function $f$ the right hand side of \eqref{eqn:lambda.logdivergence} is non-negative for $u, u' \in \Omega$, then by reversing the argument in \eqref{eqn:lambda.log.derivation} we see that $\frac{1}{\lambda} (e^{\lambda f} - 1)$ is convex.
	
	\medskip
	
	We end this section with two basic examples of our framework. The second example will be revisited in Section \ref{sec:lambda.mixture} where we introduce the $\lambda$-mixture family. 

\begin{example} [Excess growth rate] \label{eg:excess.growth}
This financial example, taken from \cite{PW16}, is the original motivation of the theory of logarithmic divergences. Consider $d \geq 2$ stocks. Over a holding period, suppose that the price of stock $i$ moves from $u_i'$ to $u_i$. We have $u, u' \in \Omega = (0, \infty)^d$. Consider a portfolio with weights $w_i$, where $w_i \geq 0$ and $\sum_{i = 1}^d w_i = 1$. Then the log return of the portfolio is given by
\[
\log \left( \sum_{i = 1}^d w_i \frac{u_i}{u_i'} \right).
\]
By Jensen's inequality, this is greater than or equal to the weighted average log return of the stocks:
\[
{\bf D}[u : u'] = \log \left( \sum_{i = 1}^d w_i \frac{u_i}{u_i'} \right) - \sum_{i = 1}^d w_i \log \frac{u_i}{u_i'} \geq 0. 
\]
We call this non-negative quantity the {\it excess growth rate} of the portfolio. Then ${\bf D}[u : u'] = {\bf L}_{-1, \varphi}[u : u']$ is the $(-1)$-logarithmic divergence of the function $\varphi(u) = - \sum_{i = 1}^d w_i \log u_i$ which is regular $c_{-1}$-convex ($\lambda = -1$) on $\Omega$.   This can be generalized to other portfolios and $\lambda < 0$; see \cite{PW16} and \cite{W19}, where the $L^{(\alpha)}$-divergence there is equivalent to our $(-\alpha)$-logarithmic divergence.
\end{example}

	\begin{example}[R\'{e}nyi entropy and divergence] \label{eg:Renyi.simplex}
		Consider the open unit simplex in $\mathbb{R}^{1 + d}$ given by
		\begin{equation} \label{eqn:unit.simplex}
			\Delta^{d} = \{ u = (u_0, u_1, \ldots, u_d) : u_i > 0, \sum_{i=0}^d u_i = 1\}.
		\end{equation}
		For $\lambda < 1$ and $\lambda \neq 0$, consider the negative (discrete) R\'{e}nyi entropy of order $q  = 1 - \lambda > 0$ given by
		\[
		\varphi_{\lambda}(p) = - {\bf H}_q^{\text{R\'{e}nyi}} (\widetilde{p}) = \frac{-1}{1 - q} \log \sum_{i = 0}^{d} \, \widetilde{p}_i^{q}, \quad p \in \Delta^{d},
		\]
		where $\widetilde{p} = \mathcal{E}_{1/q}[p]$ is the escort transformation with exponent $1/q$ (here $\nu$ is the counting measure). Then it is not difficult to verify that $\varphi_{\lambda}$ is a regular $c_{\lambda}$-convex function on $\Delta^d$ (see Figure \ref{fig:Renyi} for an illustration where $d = 1$).\footnote{We may think of $\varphi_{\lambda}$ as a function of $(p_1, \ldots, p_d)$ which takes values in an open convex set of $\mathbb{R}^d$.} The corresponding $\lambda$-logarithmic divergence is the {\it R\'{e}nyi divergence} (see \eqref{eqn:Renyi.divergence}) of the same order:
		\[
		{\bf L}_{\lambda, \varphi_{\lambda}}[ p : p'] = {\bf H}_{q}^{\text{R\'{e}nyi}}(\widetilde{p}||\widetilde{p}') = \frac{1}{q - 1} \log \sum_{i = 0}^{d} (\widetilde{p}_i)^q (\widetilde{p}_i')^{1 - q}.
		\]
		This extends the familiar fact, which we recover in the limit $\lambda \rightarrow 0$,  that the (discrete) KL-divergence is the Bregman divergence associated with the negative Shannon entropy as the convex potential function.
		
		\begin{figure}[!t]
			\centering
			\includegraphics[width=3.48in]{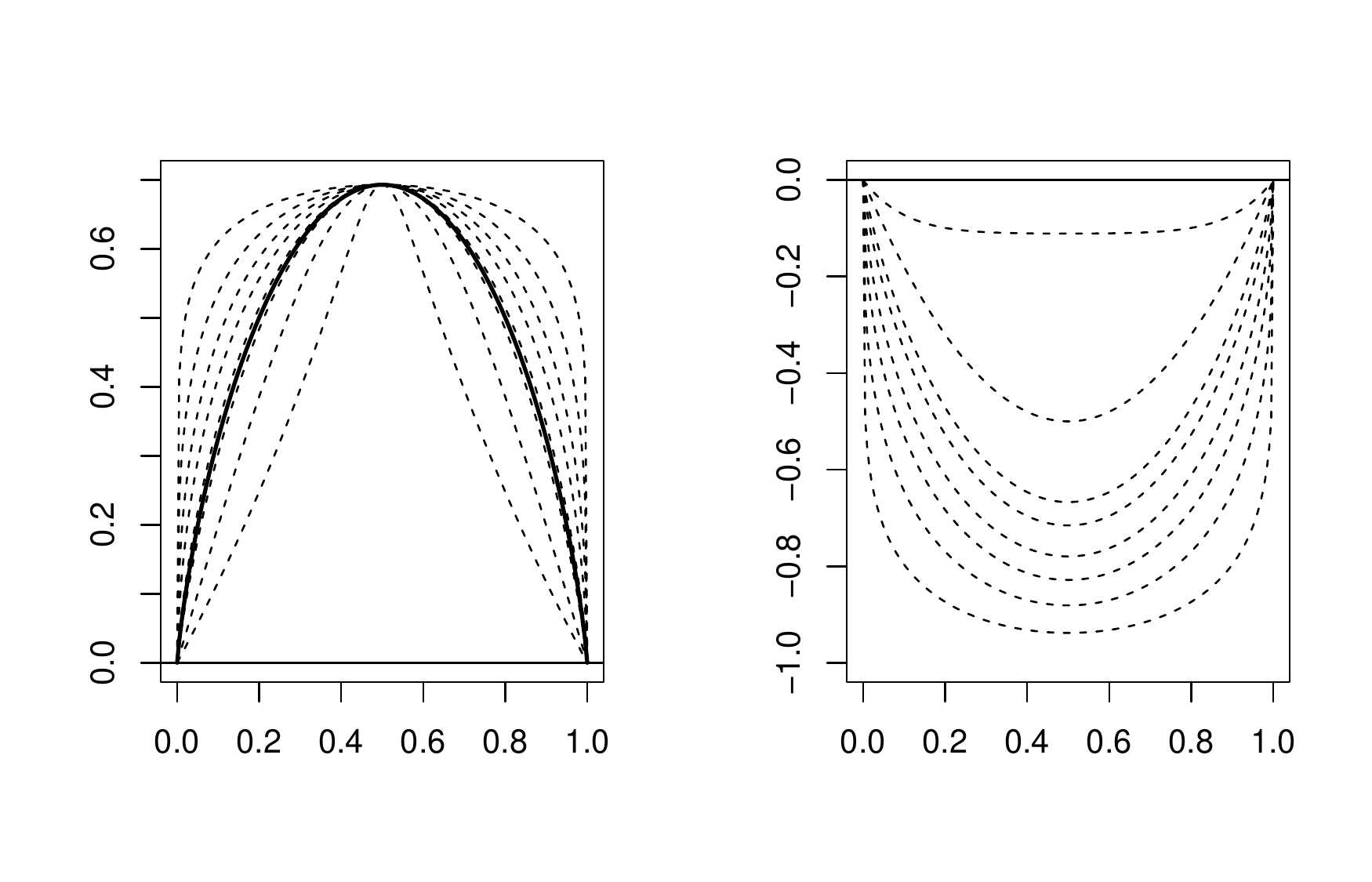}
			\vspace{-0.7cm}
			\caption{Left: The R\'{e}nyi entropy ${\bf H}_q^{\text{R\'{e}nyi}}(\widetilde{p})$ of the escort distribution $\widetilde{p}$, as a function of $p \in \Delta^1$ (identified with the unit interval), for $\lambda \in \{-5, -2, -1, -0.5, -0.1, 0.1, 0.5, 0.9\}$ and $q = 1 - \lambda$. The solid line is the Shannon entropy ($\lambda = 0$ or $q = 1$). Right: Graphs of the functions $\frac{1}{\lambda} (e^{\lambda \varphi_{\lambda}} - 1)$, which are convex, for the same values of $\lambda$. See Example \ref{eg:Renyi.simplex} which is a special case of the $\lambda$-mixture family introduced in Section \ref{sec:lambda.mixture}.}
			\label{fig:Renyi}
		\end{figure} 
	\end{example}

	\section{$\lambda$-exponential family} \label{sec:lambda.exp.family}
	In this section we introduce the $\lambda$-exponential family as an alternative representation of the $q$-exponential family tailored to the $\lambda$-duality studied in Section \ref{sec:generalized.duality}. It unifies the $\mathcal{F}^{(\pm \alpha)}$-families introduced in \cite[Section 4]{W18}. We also explain how the subtractive and divisive frameworks are related. Concrete examples will be given in Section \ref{sec:q.exp.examples}.
	
\subsection{Definition and two parameterizations}

%

\begin{definition} [$\lambda$-exponential family] \label{def:lambda.exp.family}
Let $\lambda \neq 0$ and $q = 1 - \lambda$. The $\lambda$-exponential family is the parameterized density, with respect to a given reference measure $\nu$, defined by
\begin{equation} \label{eqn:c.exponential.family} 
p_{\vartheta}(x) = p(x ; \vartheta) = e^{-c_{\lambda}(\vartheta, F(x)) - \varphi_{\lambda}(\vartheta)} = \exp_{q}(\vartheta \cdot F(x)) e^{-\varphi_{\lambda}(\vartheta)},
\end{equation}
where $F = (F_1, \ldots, F_d)$ is a vector of statistics. We also write $p_{\vartheta}(x) = p(x ; \vartheta)$. The potential function $\varphi_{\lambda}(\vartheta)$ is defined by the normalization $\int p(x ; \vartheta) d\nu(x) = 1$. Explicitly, we have
\begin{equation} \label{eqn:c.exp.normalization}
\varphi_{\lambda}(\vartheta) = \log \int e^{-c_{\lambda}(\vartheta, F(x))} \mathrm{d}\nu(x).
\end{equation}
The natural parameter set $\Omega$ is given by $\Omega = \{\vartheta \in \mathbb{R}^d : \varphi_{\lambda}(\vartheta) < \infty\}$.
\end{definition}
	
Note that we use the symbol $\vartheta$ for the ``canonical parameter'' of the $\lambda$-exponential family, and distinguish it from $\theta$ to be introduced in \eqref{eqn:vartheta.to.theta} below.

\begin{lemma} \label{lem:Omega.convex}
If $\lambda < 1$ (or $q > 0$), then the natural parameter set $\Omega$ is convex.
\end{lemma}
\begin{proof}
Note that $\Omega$ is the set of $\vartheta \in \mathbb{R}^d$ such that
\[
\int \exp_q(\vartheta \cdot F(x)) \mathrm{d}\nu(x) < \infty.
\]
Assume $q > 0$. Then $\exp_q(\vartheta \cdot F(x))$ is convex in $\vartheta$ (see the paragraph below \eqref{eqn:q.exp}). Hence $\Omega$ is a convex set.
\end{proof}

For the subsequent analysis we require some regularity conditions on the family. 
We believe some these conditions can possibly be deduced under appropriate assumptions on the family (in the spirit of the \cite{B14}). It is also of interest to study the $\lambda$-exponential family and its regularity properties from the viewpoint of nonparametric information geometry \cite{P13}. We leave this as well as a complete treatment of $\lambda$-duality to a future research.

\begin{assumption}[Regularity conditions] \label{ass:lambda.family}
Let $\lambda < 1$ (or $q = 1 - \lambda > 0$). We assume that  the natural parameter set $\Omega$ is a non-empty open convex set, and we may differentiate $\varphi_{\lambda}(\vartheta)$ under the integral sign (as many times as needed).
\end{assumption}

\begin{remark} [Interpretation of $\lambda$]
We have chosen to parameterize the function $c_{\lambda}$ in terms of $\lambda$ rather than $q = 1 - \lambda$. As will be explained in Section \ref{sec:information.geometry}, if we equip the $\lambda$-exponential family, regarded as a smooth manifold (with global coordinate system $\vartheta$), with the logarithmic divergence ${\bf L}_{\lambda, \varphi_{\lambda}}$, we obtain a statistical manifold (in the sense of information geometry \cite{a16,AN00}) with constant sectional curvature $\lambda$. Thus $\lambda$ can be thought of as the curvature parameter. On the other hand, the constant $q$ (when positive) is the order of the corresponding R\'{e}nyi entropy and divergence. Thus both $\lambda$ and $q$ are meaningful.
	\end{remark}
	
	\begin{lemma} \label{lem:positivity}
		Let $\mathbb{P}_{\vartheta}$ be a probability measure under which the random variable $X$ has density $p(\cdot ; \vartheta)$, where $\vartheta \in \Omega$ is fixed. Then $\mathbb{P}_{\vartheta}( 1 + \lambda \vartheta \cdot F(X) > 0) = 1$. 
	\end{lemma}
	\begin{proof}
		We only need to note that if $1 + \lambda \vartheta \cdot F(x) \leq 0$ then $\exp_q(\vartheta \cdot F(x))$ is either $0$ or $+\infty$.
	\end{proof}
	
	Note that when $1 + \lambda \vartheta \cdot F(x) > 0$, as in Lemma \ref{lem:positivity}, then the density \eqref{eqn:c.exponential.family} is given by
	\begin{equation} \label{eqn:lambda.exp.simplified}
		p_{\vartheta}(x) = (1 + \lambda \vartheta \cdot F(x))^{\frac{1}{\lambda}} e^{-\varphi_{\lambda}(\vartheta)}.
	\end{equation}
	Comparing \eqref{eqn:lambda.exp.simplified} and \eqref{eqn:F.alpha}, we see that the $\lambda$-exponential family coincides with the $\mathcal{F}^{(-\alpha)}$-family when $\lambda = \alpha > 0$, and is equivalent to the $\mathcal{F}^{(\alpha)}$-family when $\lambda = -\alpha < 0$ (and $F$ is replaced by $-F$).

\begin{remark} [Support condition] \label{rmk:support}
In contrast to the exponential family, it is possible that the support of the density depends on the parameter $\vartheta$; an explicit example is the $q$-Gaussian distribution discussed in Example \ref{ex:q.Guassian}. Some derivations of this paper require that the support $\{x : p(x ; \vartheta) > 0\}$ is independent of $\vartheta$. When this holds, we say that the family satisfies the {\it support condition}. \footnote{The support condition is a useful technical condition which simplifies many proofs. It may be possible to remove this assumption in some results.}
\end{remark}
	
	Note that both \eqref{eqn:c.exponential.family} and \eqref{eqn:q.exponential.family} involve the $q$-exponential function. The difference lies in the way the density is normalized. In the $q$-exponential family, the $q$-potential function $\phi_q(\theta)$ is said to be a {\it subtractive} normalization (since it is defined within $\exp_q$), while for the $\lambda$-exponential family $\varphi_{\lambda}(\vartheta)$ is given as a {\it divisive} normalization. Let us call $\phi_q(\theta)$ the {\it substractive} $q$-potential, and $\varphi_{\lambda}(\vartheta)$ the {\it divisive} $\lambda$-potential.

We show that the $\lambda$-exponential family \eqref{eqn:c.exponential.family} and $q$-exponential family \eqref{eqn:q.exponential.family} are essentially equivalent up to some changes of parameters. This involves some subtleties regarding the domains of the parameters. We first start with a $\lambda$-exponential family and rewrite it as a $q$-exponential family. 

\begin{proposition} \label{prop:lambda.as.reparamterize}
Let $\mathcal{M} = \{p_{\vartheta}(x) = \exp_{q}(\vartheta \cdot F(x)) e^{-\varphi_{\lambda}(\vartheta)} : \vartheta \in \Omega\}$ be a $\lambda$-exponential family satisfying the support condition (Remark \ref{rmk:support}). On the common support, say $\mathcal{X}_0$, and with respect to the same $F$ and dominating measure $\nu$ (restricted to $\mathcal{X}_0$), let $\mathcal{N}$ be the $q$-exponential family $\mathcal{N} = \{p_{\theta}(x) = \exp_{q}(\theta \cdot F(x) - \phi_q(\theta)): \theta \in \Theta\}$, where $\Theta$ is the natural (maximal) parameter set. 

For $\vartheta \in \Omega$, let $\theta = \vartheta e^{-\lambda \varphi_{\lambda}(\vartheta)}$. Then $\theta \in \Theta$ and $p_{\vartheta}(\cdot) = p_{\theta}(\cdot)$ on $\mathcal{X}_0$. Thus $\mathcal{M}$ may be considered as a subspace of $\mathcal{N}$. If $F_1, \ldots, F_d$ are linearly independent with respect to $\nu$ on $\mathcal{X}_0$, then the mapping $\vartheta \in \Omega \mapsto \theta = \vartheta e^{-\lambda \varphi_{\lambda}(\vartheta)} \in \Theta$ is one-to-one.
\end{proposition}
\begin{proof}
Start with \eqref{eqn:c.exponential.family} and write, for $x \in \mathcal{X}_0$, 
\begin{equation} \label{eqn:lambda.to.q.transformation}
	\begin{split}
		p_{\vartheta}(x) &= \exp_{q}(\vartheta \cdot F(x)) e^{-\varphi_{\lambda}(\vartheta)} \\
		&= \left[  1 + (1 - q) \vartheta \cdot F(x) \right]_+^{1/(1 - q)} e^{-\varphi_{\lambda}(\vartheta)} \\
		&= \left[1 + (1 - q) \left( \vartheta e^{-\lambda \varphi_{\lambda}(\vartheta)} \cdot F(x) - \frac{ e^{- \lambda \varphi_{\lambda}(\vartheta)} - 1}{-\lambda}  \right) \right]_+^{\frac{1}{1 - q}}.
	\end{split}
\end{equation}
Introduce the parameter
\begin{equation} \label{eqn:vartheta.to.theta}
\theta = \vartheta e^{-\lambda \varphi_{\lambda}(\vartheta)},
\end{equation}
and define
\begin{equation} \label{eqn:potential.change}
\widetilde{\phi}_q(\theta) = \frac{1}{-\lambda} (e^{-\lambda \varphi_{\lambda}(\vartheta)} - 1).
\end{equation}
Then, we have
\begin{equation} \label{eqn:c.exp.as.q.exp}
p_{\vartheta}(x) = \exp_{q}(\vartheta \cdot F(x)) e^{-\varphi_{\lambda}(\vartheta)} = \exp_q (\theta \cdot F(x) - \widetilde{\phi}_q(\theta)), \quad x\in \mathcal{X}_0.
\end{equation}
Since $\int_{\mathcal{X}_0} \exp_q (\theta \cdot F(x) - \widetilde{\phi}_q(\theta)) \mathrm{d}\nu(x) = 1$ by construction, we have $\theta \in \Theta$, $\tilde{\phi}_q = \phi_q$ and $p_{\vartheta} = p_{\theta}$. Note that the support condition allows us to integrate on the fixed set $\mathcal{X}_0$.

Suppose that $F_1, \ldots, F_d$ are linearly independent (with respect to $\nu$ on $\mathcal{X}_0$). Then the mappings $\vartheta \in \Omega \mapsto p_{\vartheta} \in \mathcal{M}$ are $\theta \in \Theta \mapsto p_{\theta} \in \mathcal{N}$ are one-to-one, i.e., the two models are uniquely identifiable. Given distinct $\vartheta, \vartheta' \in \Omega$, the densities $p_{\vartheta}$ and $p_{\vartheta'}$ define distinct probability distributions. But $p_{\vartheta} = p_{\theta}$ and $p_{\vartheta'} = p_{\theta'}$, where $\theta = \vartheta e^{-\lambda \varphi_{\lambda}(\vartheta)}$ and $\theta' = \vartheta' e^{-\lambda \varphi_{\lambda}(\vartheta')}$. Thus $\theta$ and $\theta'$ are also distinct elements of $\Theta$. 
\end{proof}

Next we want to express a given $q$-exponential family as a $\lambda$-exponential family. Note that reversing in a direct manner the argument in \eqref{eqn:lambda.to.q.transformation} requires that $1 - \lambda \phi_q(\theta) > 0$. So, unfortunately, the converse of Proposition \ref{prop:lambda.as.reparamterize} does not hold in general. Nevertheless, we show that a reparameterization can be done {\it locally} by possibly redefining $F$. Here is a precise statement.
 
\begin{proposition} \label{prop:q.to.lambda}
Consider a $q$-exponential family $\{p_{\theta}(x) = \exp_{q}(\theta \cdot F(x) - \phi_q(\theta)): \theta \in \Theta\}$. Let $\theta_0$ be in the interior of $\Theta$. If $1 - \lambda \phi_q(\theta_0) > 0$ or $\theta_0 \neq 0$, there exists a neighborhood $U$ of $\theta_0$ such that each $p_{\theta}$ for $\theta \in U$ can be written in the form
\begin{equation} \label{eqn:q.exp.as.lambda.exp}
p_{\theta}(x) = \exp_q(\vartheta \cdot \widetilde{F}(x)) e^{-\varphi_{\lambda}(\vartheta)},
\end{equation}
where $\vartheta$ is a function of $\theta$, $\widetilde{F}(x) = F(x) - c \theta_0$ and $c \in \mathbb{R}$ is a constant.
\end{proposition}
\begin{proof}
We first consider the case that $1 - \lambda \phi_q(\theta_0) > 0$. Then there exists a neighborhood $U$ such that $1 - \lambda \phi_q(\theta) > 0$ for $\theta \in U$. Reversing the computation in \eqref{eqn:lambda.to.q.transformation}, we have, for $\theta \in U$,
\begin{equation*}
\begin{split}
	p_{\theta}(x) &= \left[ 1 + \lambda (\theta \cdot F(x) - \phi_q(\theta)) \right]_+^{\frac{1}{1 - q}} \\
	&= (1 - \lambda \phi_q(\theta))^{\frac{1}{\lambda}} \left[ 1 + \lambda \frac{\theta}{1 - \lambda \phi_q(\theta)} \cdot F(x) \right]_+^{\frac{1}{\lambda}}.
\end{split}
\end{equation*}
Writing
\begin{equation} \label{eqn:theta.to.vartheta}
\vartheta =  \frac{\theta}{1 - \lambda \phi_q(\theta)} \quad \text{and} \quad \varphi_{\lambda}(\vartheta) = \frac{1}{-\lambda} \log (1 - \lambda \phi_q(\theta)),
\end{equation}
we have $p_{\theta}(x) = \exp_q(\vartheta \cdot F(x)) e^{-\varphi_{\lambda}(\vartheta)}$ which is in the form of a $\lambda$-exponential family (where we may pick $c = 0$).

Next suppose $1 - \lambda \phi_q(\theta_0)  \leq 0$ and $\theta_0 \neq 0$. Define $\widetilde{F}(x) = F(x) - c\theta_0$, where $c \in \mathbb{R}$ is a constant to be chosen. Fix $\epsilon > 0$. Let $U$ be a neighborhood of $\theta_0$ on which $\theta \cdot \theta_0 > \frac{1}{2} |\theta_0|^2$ and $1 - \lambda \phi(\theta) > 1 - \lambda \phi_q(\theta_0) - \varepsilon$. For $\theta \in U$, write
\begin{equation*}
\begin{split}
1 - \lambda \phi_q(\theta) + \lambda \theta \cdot F(x) &= 1 - \lambda \phi(\theta) + \lambda c \theta \cdot \theta_0 + \lambda \theta \cdot \widetilde{F}(x).
\end{split}
\end{equation*}
Choose $c \in \mathbb{R}$ such that $c\lambda > 0$ and $1 - \lambda \phi(\theta_0) - \epsilon + c \frac{\lambda}{2} |\theta_0|^2 > 0$ (this requires $\theta_0 \neq 0$). Then, for $\theta \in U$, we have
\[
1 - \lambda \phi(\theta) + \lambda c \theta \cdot \theta_0 > 0.
\]
Now we may write
\begin{equation*}
\begin{split}
p_{\theta}(x) &= \left( 1 - \lambda \phi(\theta) + \lambda c \theta \cdot \theta_0 \right)^{\frac{1}{\lambda}} \left[ 1 + \lambda \frac{\theta}{1 - \lambda \phi(\theta) + \lambda c \theta \cdot \theta_0} \cdot \widetilde{F}(x)  \right]_+^{\frac{1}{\lambda}},
\end{split}
\end{equation*}
which has the form \eqref{eqn:q.exp.as.lambda.exp} if we let
\[
\vartheta = \frac{\theta}{1 - \lambda \phi(\theta) + \lambda c \theta \cdot \theta_0} \quad \text{and} \quad \varphi_{\lambda}(\vartheta) = \frac{1}{-\lambda} \log (1 - \lambda \phi(\theta) + \lambda c \theta \cdot \theta_0).
\]
\end{proof}

Roughly speaking, the above results say that when studying local properties of the family it does not matter whether we use the subtractive ($q$-exponential family) or divisive ($\lambda$-exponential family) formulations.

	\subsection{Linking to R\'{e}nyi entropy and R\'{e}nyi divergence} \label{sec:connetions.with.Renyi}
	While divisive normalizations of the $q$-exponential family had been considered before (see for example (7.13) in \cite[Section 7.3]{N11}; a similar generalized exponential family is considered in \cite{KM20}), its theoretical significance was not recognized because it was not paired with the $\lambda$-duality. Here, we show that the $\lambda$-duality, made possible by the following result, offers fresh insights into the family and leads naturally to the R\'{e}nyi entropy and divergence.

\begin{theorem} \label{thm:c.exp.potential}
Consider a $\lambda$-exponential family satisfying Assumption \ref{ass:lambda.family}. Let $\varphi_{\lambda}$ be the divisive $\lambda$-potential. Then the function $\frac{1}{\lambda} (e^{\lambda \varphi_{\lambda}(\vartheta)} - 1)$ is convex on $\Omega$. Moreover, we have that $1 - \lambda \nabla \varphi(\vartheta) \cdot \vartheta  > 0$.
	\end{theorem}
	\begin{proof}
		The proof of the first statement can be found in Propositions 2 and 3 of \cite{W18} where the results are stated in terms of the $\mathcal{F}^{(\pm \alpha)}$-families. To prove the second statement, consider
		\[
		e^{\varphi(\vartheta)} = \int [1 + \lambda \vartheta \cdot F]_+^{1/\lambda} \mathrm{d}\nu.
		\]
		Differentiating under the integral sign, which is possible by assumption, we have
		\begin{equation*}
			\begin{split}
				& e^{\varphi(\vartheta)} \nabla \varphi(\vartheta) = \int [1 + \lambda \vartheta \cdot F]_+^{1/\lambda - 1} F \mathrm{d}\nu \\
				& \Rightarrow \nabla \varphi(\vartheta) = \int p(x; \theta) \frac{F}{1 + \lambda \vartheta \cdot F} \mathrm{d}\nu,
			\end{split}
		\end{equation*}
		where the second line follows from Lemma \ref{lem:positivity}. We get
		\begin{equation*}
			\begin{split}
				1 - \lambda \nabla \varphi(\vartheta) \cdot \vartheta &= \int p(x; \theta) \left( 1 - \frac{\lambda \vartheta \cdot F}{1 + \lambda \vartheta \cdot F} \right) \mathrm{d}\nu \\
				&= \int p(x ; \theta) \frac{1}{1 + \lambda \vartheta \cdot F} \mathrm{d}\nu \\
				&= \int p(x ; \theta) \frac{1}{[1 + \lambda \vartheta \cdot F]_+} \mathrm{d}\nu > 0.
			\end{split}
		\end{equation*}
	\end{proof}
	
\begin{condition} \label{con:technical.conditions}
In the remainder of Section \ref{sec:lambda.exp.family} we assume that Assumption \ref{ass:lambda.family} holds and that the Hessian of $\frac{1}{\lambda} (e^{\lambda \varphi_{\lambda}(\vartheta)} - 1)$ is strictly positive definite. This implies that $\varphi_{\lambda}$ is a regular $c_{\lambda}$-convex function and Theorem \ref{thm:lambda.duality} applies. We also assume that the support condition (see Remark \ref{rmk:support}) holds and that $(F_1, \ldots, F_d)$ are linearly independent.
\end{condition}
	
	By Proposition \ref{prop:lambda.as.reparamterize}, there are two potential functions, namely $\phi_q(\theta)$ and $\varphi_{\lambda}(\vartheta)$, associated to the density. The two potential functions and their respective dualities define apparently two dual variables, namely
	\[
	\eta_{\mathrm{subtractive}} = \nabla_{\theta} \phi_q(\theta) \text{ and } \eta_{\mathrm{divisive}} = \nabla_{\vartheta}^{c_{\lambda}} \varphi_{\lambda}(\vartheta).
	\]
	For clarify, we sometimes use $\nabla_u$ to denote the gradient with respect to the variable $u$. We show that the two dual variables are actually the same. 
	
\begin{lemma}
The mapping $\vartheta \in \Omega \mapsto \theta = \vartheta e^{-\lambda \varphi_{\lambda}(\vartheta)}$ is a diffeomorphism from $\Omega$ onto its range.
\end{lemma}	
\begin{proof}
Note that Condition \ref{con:technical.conditions} is in force. Clearly the mapping $\vartheta \mapsto \theta = \vartheta e^{-\lambda \varphi_{\lambda}(\vartheta)}$ is differentiable. By Proposition \ref{prop:lambda.as.reparamterize}, it is also one-to-one. Recall that the gradient is regarded as a column vector. By a direct differentiation, we see that the Jacobian is given by
\[
\frac{\partial \theta}{\partial \vartheta}(\vartheta) = e^{-\lambda \varphi_{\lambda}(\theta)} \left( {\bf I}_d - \lambda \vartheta (\nabla_{\vartheta} \varphi_{\lambda}(\vartheta))^{\top}\right),
\]
where ${\bf I}_d$ is the identity matrix. Since $1 - \lambda \nabla_{\vartheta} \varphi(\vartheta) \cdot \vartheta > 0$ by assumption, by the Shermon-Morrison formula we can invert the Jacobian is invertible:
\begin{equation} \label{eqn:Jacobian.inverse}
\left( \frac{\partial \theta}{\partial \vartheta}(\vartheta) \right)^{-1} = e^{\lambda \varphi_{\lambda}(\vartheta)} \left( {\bf I}_d + \frac{\lambda \vartheta (\nabla_{\vartheta} \varphi_{\lambda}(\vartheta))^{\top}}{1 - \lambda \nabla_{\vartheta} \varphi_{\lambda}(\vartheta) \cdot \vartheta} \right).
\end{equation}
By the inverse function theorem, the mapping $\vartheta \mapsto \theta$ is a diffeomorphism.
\end{proof}

\begin{theorem} \label{thm:c.exp.dual.variable}
We have
\begin{equation} \label{eqn:c.exp.dual.variable}
\nabla_{\vartheta}^{c_{\lambda}} \varphi_{\lambda}(\vartheta) = \nabla_{\theta} \phi_q(\theta) =: \eta.
\end{equation}	
Thus, under both the subtractive and divisive frameworks, the dual variable $\eta$ is the escort expectation \eqref{eqn:escort}.
\end{theorem}
\begin{proof}  
Recall that the gradient is regarded as a column vector. By the lemma above, $\theta$ is a function of $\vartheta$. Applying the chain rule to \eqref{eqn:potential.change}, we have
\[
(\nabla_{\theta} \phi_q(\theta))^{\top} =  e^{-\lambda \varphi_{\lambda}(\vartheta)} (\nabla_{\vartheta} \varphi_{\lambda}(\vartheta) )^{\top} \frac{\partial \vartheta}{\partial \theta}(\theta).
\]
Note that $\frac{\partial \vartheta}{\partial \theta}(\theta)$ is given by \eqref{eqn:Jacobian.inverse}. Plugging this into the above and using \eqref{eqn:c.Legendre.transform}, we compute
\begin{equation*}
\begin{split}
(\nabla_{\theta} \phi_q(\theta))^{\top} 
&= (\nabla_{\vartheta}\varphi_{\lambda}(\vartheta) )^{\top}  \left( {\bf I}_d + \frac{\lambda \vartheta (\nabla_{\vartheta} \varphi_{\lambda}(\vartheta))^{\top}}{1 - \lambda \nabla_{\vartheta} \varphi_{\lambda}(\vartheta) \cdot \vartheta} \right) \\
&= (\nabla_{\vartheta}\varphi_{\lambda}(\vartheta) )^{\top} \left( 1 + \frac{\lambda \nabla_{\vartheta} \varphi(\vartheta) \cdot \vartheta}{1 - \lambda \nabla_{\vartheta} \varphi_{\lambda}(\vartheta) \cdot \vartheta} \right) \\
&= \frac{(\nabla_{\vartheta}\varphi_{\lambda}(\vartheta) )^{\top}}{1 - \lambda \nabla_{\vartheta} \varphi_{\lambda}(\vartheta) \cdot \vartheta} \\
&= (\nabla_{\vartheta}^{c_{\lambda}} \varphi_{\lambda}(\vartheta))^{\top}.
\end{split}
\end{equation*}
Thus $\nabla_{\vartheta}^{c_{\lambda}} \varphi_{\lambda}(\vartheta) = \nabla_{\theta} \phi_q(\theta)$ and the theorem is proved.
\end{proof}

	From \eqref{eqn:escort}, the dual variable $\eta$ is the expected value of $F(X)$ under the escort distribution $\widetilde{p}_{\vartheta}$ (with exponent $q$). In Theorem \ref{thm:dual.as.barycenter} below, we give under suitable conditions a new probabilistic interpretation of $\eta$ in terms of the {\it original} distribution $p_{\vartheta}$ without undergoing the escort transformation. Note that although the dual parameter $\eta$ remains the same, the primal variables $\theta$ (subtractive case) and $\vartheta$ (divisive case) are different and are related by \eqref{eqn:vartheta.to.theta}.
	
	\begin{remark}[Geometric meanings of the coordinate systems $\vartheta$ and $\eta$] \label{rmk:primal.dual.coordinates}
		It is helpful to think of the primal coordinate system $\vartheta$ of a $\lambda$-exponential family as a projective affine coordinate system, in the sense that a $\vartheta$-straight line is the trajectory of a primal geodesic (up to a time reparameterization). Similarly, the dual coordinate system is the dual projective affine coordinate system. These notions are justified by the information geometry of $\lambda$-logarithmic divergence described in Section \ref{sec:information.geometry}.
	\end{remark}
	
	As explained in Section \ref{sec:Legendre.duality}, under the classical Legendre duality, the exponential family leads naturally to the Shannon entropy ${\bf H}(\cdot)$ and the KL-divergence ${\bf H}( \cdot || \cdot)$. For a $\lambda$-exponentialy family under the $\lambda$-duality, the natural objects are the R\'{e}nyi entropy and R\'{e}nyi divergence. The R\'{e}nyi entropy was defined in \eqref{eqn:Renyi.entropy}. Recall that the {\it R\'{e}nyi divergence} of order $q > 0$ between probability densities $p_1, p_2$ with respect to $\nu$ is defined by 
	\begin{equation} \label{eqn:Renyi.divergence}
		\begin{split}
			{\bf H}^{\mathrm{R\acute{e}nyi}}_q(p_1 || p_2) &= \frac{1}{q - 1} \log \int p_1(x)^{q} p_2(x)^{1 - q} \mathrm{d}\nu(x).
		\end{split}
	\end{equation}
	See \cite{VH14} for a useful overview of the properties of the R\'{e}nyi entropy and divergence. 
	
	The following theorem illustrates the theoretical elegance of the $\lambda$-duality. Here, the Shannon entropy (the negative conjugate function for the exponential family) is replaced by the R\'{e}nyi entropy, and the KL-divergence (Bregman divergence of the potential function) is replaced by the R\'{e}nyi entropy. Recall that Condition \ref{con:technical.conditions} is imposed.

	\begin{theorem} \label{thm:Renyi.divergence}
		The $\lambda$-logarithmic divergence of the divisive $\lambda$-potential $\varphi_{\lambda}$ is the R\'{e}nyi divergence:
		\begin{equation} \label{eqn:L.div.as.Renyi}
			{\bf L}_{\lambda, \varphi_{\lambda}}[ \vartheta : \vartheta'] = {\bf H}^{\mathrm{R\acute{e}nyi}}_q( p_{\vartheta'} || p_{\vartheta} ).
		\end{equation}
		Moreover, the $\lambda$-conjugate of $\varphi_{\lambda}$ is given on the range of $\nabla^{c_{\lambda}} \varphi_{\lambda}$ by the negative R\'{e}nyi entropy:
		\begin{equation} \label{eqn:conjugate.as.Renyi}
			\psi_{\lambda}(\eta) = - {\bf H}^{\mathrm{R\acute{e}nyi}}_q( p_{\vartheta}), \quad \eta = \nabla^{c_{\lambda}} \varphi_{\lambda}(\vartheta).
		\end{equation}
	\end{theorem}
	\begin{proof}
		See \cite[Theorem 13]{W18}.
	\end{proof}
	
	
	\subsection{R\'{e}nyi entropy maximization}  \label{sec:entropy.maximize}
	Let $q > 0$ and a let a reference measure $\nu$ be given. Let $F = (F_1, \ldots, F_d)$ be a vector of statistics. Consider the R\'{e}nyi entropy maximization problem
	\begin{equation} \label{eqn:max.entropy}
		\max_{P \sim \nu} {\bf H}^{\text{R\'{e}nyi}}_q (P) \text{ subject to } \widetilde{\mathbb{E}}_P [F(X)] = y,
	\end{equation}
	where $P$ is a probability measure, $\widetilde{\mathbb{E}}_P$ is the expectation with respect to the escort distribution $\widetilde{p} = \mathcal{E}_q [p]$ and $p = \frac{dP}{d\nu}$. Last but not least, the constraint $P \sim \nu$ means that the $P$ and $\nu$ are equivalent (thus $p > 0$ $\nu$-a.e.); intuitively, we assume that the support of the distribution $P$ is known to be that of $\nu$. Since the Tsallis and R\'{e}nyi entropies are monotonic transformations of each other (see \eqref{eqn:Tsallis.Renyi}), problem \eqref{eqn:max.entropy} remains the same if we maximize instead the Tsallis entropy of order $q$. It is well-known that the solution to \eqref{eqn:max.entropy} can be written in the form of a $q$-exponential family; see for example the $q$-Max-Ent Theorem in \cite{AO11}. The usual proof of this result uses Lagrange multipliers. Here we give a new proof which utilizes the $\lambda$-exponential family and the associated $\lambda$-logarithmic divergence. By letting $\lambda \rightarrow 0$ we recover the entropy maximizing property of the exponential family.
	
	\begin{theorem} \label{thm:Renyi.entropy.maximize}
	With the given $F$ and $\nu$, consider the $\lambda$-exponential family \eqref{eqn:c.exponential.family} and suppose for some parameter $\vartheta^* \in \Omega$ we have $\widetilde{\mathbb{E}}_{\vartheta^*} [ F(X)] = y$. Then the distribution $P^*$ with $\frac{\mathrm{d} P^*}{\mathrm{d}\nu} = p_{\vartheta^*}$ is the unique solution to the R\'{e}nyi entropy maximization problem \eqref{eqn:max.entropy}.
	\end{theorem}
	\begin{proof}
		Let $P \sim \nu$ be a distribution which satisfies the constraint on the escort expectation, and let $p = \frac{dP}{d\nu} > 0$ be its density.
		
		Consider the R\'{e}nyi divergence
		\begin{equation*}
			\begin{split}
				{\bf H}^{\text{R\'{e}nyi}}_q(P || P_{\vartheta^*}) &= \frac{1}{q - 1} \log \int p^{q} p_{\vartheta^*}^{1 - q} \mathrm{d}\nu.
			\end{split}
		\end{equation*} 
		Since $\lambda = 1 - q$, by \eqref{eqn:c.exponential.family}  and Lemma \ref{lem:positivity}, we have, $\nu$-a.e., 
		\begin{equation*}
			\begin{split}
				p_{\vartheta^*}^{1 - q} (x) &= [1 + \lambda \vartheta^* \cdot F(x)]^{}_+ \, e^{-\lambda \varphi_{\lambda}(\vartheta^*)} \\
				&= [1 + \lambda \vartheta^* \cdot F(x)] \, e^{-\lambda \varphi_{\lambda}(\vartheta^*)}.
			\end{split}
		\end{equation*}
		Now we compute
		\begin{equation} \label{eqn:Renyi.div.computation}
			\begin{split}
				&{\bf H}^{\text{R\'{e}nyi}}_q(P || P_{\vartheta^*}) 
				= \frac{-1}{\lambda}  \log \left( \int p^{q} \, (1 + \lambda \vartheta^* \cdot F) \, \mathrm{d}\nu\right) + \varphi_{\lambda}(\vartheta^*) \\
				&= \frac{-1}{\lambda} \log \int p^{q} \mathrm{d}\nu - \frac{1}{\lambda} \log \left( 1 + \lambda \vartheta^* \cdot \int \frac{p^{q}}{\int p^{q }  \mathrm{d}\nu} F \mathrm{d}\nu \right) + \varphi_{\lambda}(\vartheta^*) \\
				&= - {\bf H}^{\text{R\'{e}nyi}}_q (P) - \frac{1}{\lambda} \log (1 + \alpha \vartheta^* \cdot y) + \varphi_{\lambda}(\vartheta^*).
			\end{split}
		\end{equation}		
		Note that in the last equality we used the assumption $\widetilde{\mathbb{E}}_P [ F(X)] = y$.
		
		On the other hand, by construction we have
		\begin{equation*}
			\begin{split}
				&c = \int F \frac{p_{\vartheta^*}^{q}}{\int p_{\vartheta^*}^{q} \mathrm{d}\nu} \mathrm{d}\nu \\
				&\Rightarrow 1 + \lambda \vartheta^* \cdot y = \int (1 + \lambda \vartheta^* \cdot F) \frac{p_{\vartheta^*}^{q}}{\int p_{\vartheta^*}^{q} \mathrm{d}\mu} \mathrm{d}\mu. 
			\end{split}
		\end{equation*}
		Taking logarithm and rearranging, we have
		\begin{equation*}
			\begin{split}
				&\log(1 + \lambda \vartheta^* \cdot y) = \log \int (1 + \lambda \vartheta^* \cdot F) p_{\vartheta^*}^{q} \mathrm{d}\nu - \log \int p_{\vartheta^*}^{q} \mathrm{d}\nu \\
				&= \log \int p_{\theta^*} e^{\lambda \varphi_{\lambda}(\vartheta^*)} \mathrm{d}\nu - \log \int p_{\vartheta}^{1 + \alpha} \mathrm{d}\nu \\
				&= \lambda \varphi_{\lambda}(\vartheta^*) - \log \int p_{\vartheta^*}^{q} \mathrm{d}\nu.
			\end{split}
		\end{equation*}
		It follows that
		\[
		{\bf H}^{\text{R\'{e}nyi}}_q(P_{\vartheta^*}) = \frac{-1}{\lambda} \log (1 + \lambda \vartheta^* \cdot y) + \varphi_{\lambda}(\vartheta^*).
		\]
		Plugging into \eqref{eqn:Renyi.div.computation}, we have
		\begin{equation*} 
			{\bf H}^{\text{R\'{e}nyi}}_q(P_{\vartheta^*}) - {\bf H}^{\text{R\'{e}nyi}}_q(P) = {\bf H}^{\text{R\'{e}nyi}}_q(P || P_{\vartheta^*}) \geq 0.
		\end{equation*}
		Thus ${\bf H}^{\text{R\'{e}nyi}}_q(P_{\vartheta^*}) \geq {\bf H}^{\text{R\'{e}nyi}}_q(P)$ and equality holds only if $P = P_{\vartheta^*}$. This completes the proof of the theorem.
	\end{proof}
	
	Observe that the proof of the theorem reveals something more. Not only do we get the  R\'{e}nyi entropy maximizing property of the $\lambda$-exponential family, we also obtain an {\it explicit} expression of the optimality gap in terms of the R\'{e}nyi divergence:
	\begin{equation} \label{eqn:Renyi.divergence.add}
		{\bf H}^{\text{R\'{e}nyi}}_q(P_{\vartheta^*}) = {\bf H}^{\text{R\'{e}nyi}}_q(P) + {\bf H}^{\text{R\'{e}nyi}}_q(P || P_{\vartheta^*}).
	\end{equation}
	Thus the R\'{e}nyi divergence equals the difference of two R\'{e}nyi entropies. To the best of our knowledge, this result is new. We also recall that the R\'{e}nyi divergence is additive under independent sampling: Given two product probability measures $P_1 \otimes P_2$ and $Q_1 \otimes Q_2$, we have
	\begin{equation} \label{eqn:Renyi.divergence.additive}
		\begin{split}
			&{\bf H}^{\mathrm{R\acute{e}nyi}}_q(P_1 \otimes P_2 || Q_1 \otimes Q_2) \\
			&= {\bf H}^{\mathrm{R\acute{e}nyi}}_q(P_1 || Q_1) + {\bf H}^{\mathrm{R\acute{e}nyi}}_q(P_2 || Q_2).
		\end{split}
	\end{equation}
	Equations \eqref{eqn:Renyi.divergence.add} and \eqref{eqn:Renyi.divergence.additive} show the advantage of using R\'{e}nyi entropy and divergence (which correspond to the $\lambda$-duality) over the Tsallis entropy and divergence. We also mention the recent paper \cite{MR21} which studies R\'{e}nyi entropy maximization using the $L^{(\alpha)}$-divergence of \cite{W18} (which is the same as the $(-\alpha)$-logarithmic divergence) and orthogonal foliations of statistical manifolds. The foliation is studied independently in \cite{TW21} and applied in nonlinear principal component analysis. 
	
	\subsection{Examples} \label{sec:q.exp.examples}
	In this subsection we illustrate the framework with some representative probability distributions. In particular, we observe that the {\it $\alpha$-family} studied by Amari and Nagaoka \cite{AN00} can be viewed as a special case of the $\lambda$-exponential family (and hence the $q$-exponential family). 
	
	\begin{example} [Cauchy location-scale family] \label{ex:Cauchy}
		Consider the Cauchy location-scale family whose density on $\mathbb{R}$ (with respect to the Lebesgue measure) is given by
		\begin{equation} \label{eqn:Cauchy}
			p(x; \mu, \sigma) = \frac{1}{\pi} \frac{\sigma}{\sigma^2 + (x - \mu)^2},
		\end{equation}
		where $(\mu, \sigma) \in \mathbb{R} \times (0, \infty)$. Define $F(x) = (x, x^2)$. Rearranging \eqref{eqn:Cauchy}, we may write
		\begin{equation}
			p(x; \mu, \sigma) = p(x ; \vartheta) 
			:= \frac{1}{1 - \vartheta \cdot F(x)}\frac{1}{\pi} \sqrt{ -\vartheta_2 -  \frac{\vartheta_1^2}{4}   },
		\end{equation}
		where
		\[
		\vartheta = (\vartheta_1, \vartheta_2) = \left( \frac{2\mu}{\mu^2 + \sigma^2}, \frac{-1}{\mu^2 + \sigma^2} \right).
		\]
		Equivalently, we have $\mu = \frac{-\vartheta_1}{2\vartheta_2}$ and $\sigma^2 = \frac{-4\vartheta_2 - \vartheta_1^2}{4\vartheta_2^2}$. This expresses the Cauchy family as a $\lambda$-exponential family with $\lambda = -1$. The divisive potential function is given by
		\begin{equation} \label{eqn:cauchy.potential}
			\varphi_{-1}(\vartheta) = \frac{-1}{2} \log \left(-\vartheta_2 -  \frac{\vartheta_1^2}{4}  \right) + \log \pi,
		\end{equation}
		and the natural parameter set is $\Omega = \{\vartheta \in \mathbb{R}^2 :  \vartheta_1^2 + 4 \vartheta_2<0 \}$ whose boundary is a parabola. By a direct computation using \eqref{eqn:cauchy.potential}, it can be shown that the dual variable is given by
		\[
		\eta = \nabla^{c_{-1}} \varphi_{-1}(\vartheta) = \left( \frac{-\vartheta_1}{2 \vartheta_2}, \frac{-1}{\vartheta_2} \right).
		\]
		Similarly, for any fixed degree of freedom, Student's $t$-distribution parameterized by location and scale can be expressed as a $\lambda$-exponential family; see Example \ref{eg:t.dist}. 
	\end{example}
	
	\begin{example} [$q$-Gaussian distribution] \label{ex:q.Guassian}
		For $q < 3$, the (centered) $q$-Gaussian distribution on the real line is defined by the parameterized density
		\begin{equation} \label{eqn:q.Gaussian}
			f(x ; \vartheta) = C_q \sqrt{\vartheta} \exp_q(-\vartheta x^2), \quad x \in \mathbb{R},
		\end{equation}
		where $\nu$ is the Lebesgue measure, $\vartheta \in \Omega = (0, \infty)$ and $C_q > 0$ is an explicit constant depending only on the value of $q$ (see for example \cite[(7.41)]{N11}). When $q < 1$ the density is supported on the finite interval $|x| < \sqrt{\vartheta/(1 - q)}$ which depends on $\vartheta$, so the support condition is violated. For $1 < q < 3$, the density is a scaled and reparameterized version of the $t$-distribution with $\frac{3 - q}{q - 1}$ degrees of freedom. When $q = 1$, it reduces to the normal distribution $N(0, \frac{1}{2})$. When $q \geq 3$ the density cannot be normalized. It is helpful to consider $\sigma^2 = 1/\vartheta$ where $\sigma > 0$ is the scale parameter.
		
		Comparing \eqref{eqn:q.Gaussian} with \eqref{eqn:c.exponential.family}, we see that \eqref{eqn:q.Gaussian} is a $\lambda$-exponential family with $\lambda = 1 - q$, $\vartheta \in (0, \infty)$ and $F(x) = -x^2$. The divisive $\lambda$-potential is given by
		\begin{equation} \label{eqn:q.exp.potential}
			\varphi_{\lambda}(\vartheta) = \frac{-1}{2} \log \vartheta  + C_{\lambda}',
		\end{equation}
		where $C_{\lambda}' = \log C_q$. It is easy to verify that $\varphi_{\lambda}$ is regular $c_{\lambda}$-convex for all $q < 3$ (or $\lambda > -2$). In particular, for $\lambda > -2$ we have $1 - \lambda \varphi_{\lambda}'(\vartheta) \vartheta = 1 + \frac{\lambda}{2} > 0$.
		
		The $\lambda$-logarithmic divergence is given for $\vartheta, \vartheta' \in (0, \infty)$ by 
		\begin{equation} \label{eqn:q.Gaussian.div}
			{\bf L}_{\lambda, \varphi_{\lambda}} [ \vartheta : \vartheta'] = \frac{1}{2} \log \frac{\vartheta'}{\vartheta} - \frac{1}{\lambda} \log \left( 1 - \frac{\lambda}{2} \left( \frac{\vartheta}{\vartheta'} - 1\right) \right).
		\end{equation}
		See Figure \ref{fig:q.Gaussian} for a graphical illustration of this divergence for several values of $\lambda$. Note that the divergence may take value $+\infty$ when $\lambda > 0$.
		
		\begin{figure}[!t]
			\centering
			\includegraphics[width=3.5in]{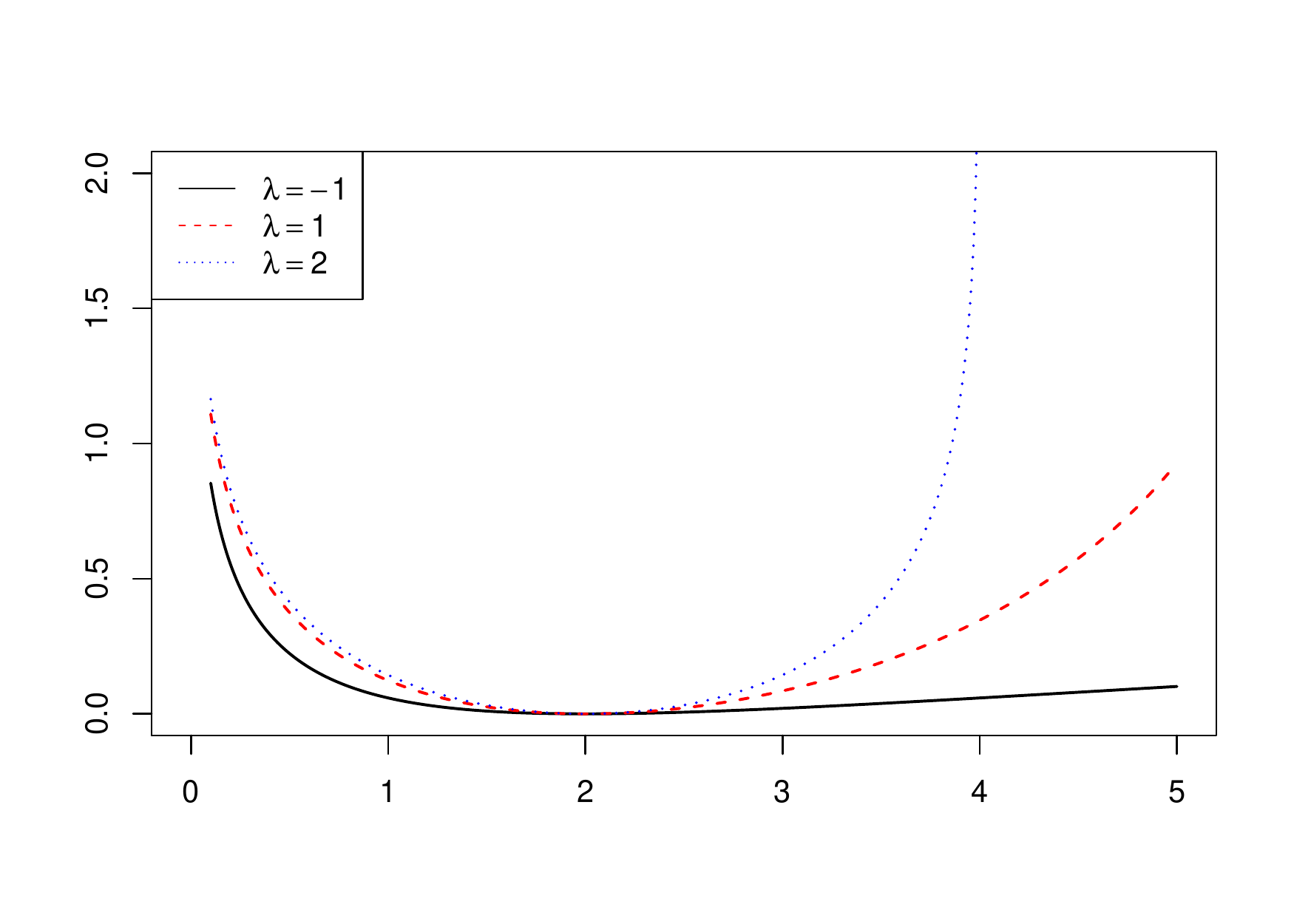}
			\vspace{-0.8cm}
			\caption{The $\lambda$-logarithmic divergence $\vartheta \mapsto {\bf L}_{\lambda, \varphi_{\lambda}} [\vartheta : \vartheta_0]$ of the $q$-Gaussian distribution (see \eqref{eqn:q.Gaussian.div}), as a function of the first variable $\vartheta \in (0, \infty)$, for several values of $\lambda$. Here $\vartheta_0 = 2$.}
			\label{fig:q.Gaussian}
		\end{figure}
		
		The dual parameter is given by
		\[
		\eta = \nabla^{c_{\lambda}} \varphi_{\lambda}(\vartheta) =\frac{-1}{2 + \lambda} \frac{1}{\vartheta} = \frac{-1}{3 - q} \sigma^2.
		\]
		Note that we have a minus sign because $F(x) = -x^2$. This is consistent with the known value $\widetilde{\mathbb{E}}_{\vartheta} [X^2] = \frac{1}{3 - q} \sigma^2$ (see for example \cite[Problem 7.6]{N11}). Note that here this formula holds even when $q < 0$.		
		
		Next consider the $\lambda$-conjugate of $\varphi_{\lambda}$ given by
		\[
		\psi_{\lambda}(\eta) = \sup_{\vartheta' > 0} \left\{ \frac{1}{\lambda} \log(1 + \lambda \vartheta' \eta) - \left( \frac{-1}{2} \log \vartheta'  + C_{\lambda}' \right) \right\}.
		\]
		We may optimize over $\vartheta > 0$ because $\frac{-1}{2} \log \vartheta = \infty$ for $\vartheta \leq 0$. It can be verified that the $\lambda$-conjugate, which is the R\'{e}nyi entropy by Theorem \ref{thm:Renyi.divergence}, is given by
		\begin{equation} \label{eqn:q.exp.dual.function}
			\psi_{\lambda}(\eta) = \frac{-1}{2} \log (-\eta) + C_{\lambda}'',
		\end{equation}
		where $C_{\lambda}''$ is a constant depending on $\lambda$. Note that \eqref{eqn:q.exp.potential} and \eqref{eqn:q.exp.dual.function} have the same form other than a change of sign. 
	\end{example}
	
	Another fundamental example of $\lambda$-exponential family is the {\it Dirichlet perturbation model}. It was used in \cite{PW18b} to study the  Dirichlet optimal transport problem whose solution utilizes the logarithmic duality used in this paper. According to \cite{D68}, this distribution -- also called the {\it shifted Dirichlet distribution} -- was considered by Savage as early as 1966. To the best of our knowledge, it has not been considered in the context of $q$-exponential family (also see \cite{RT14}).
	
	\begin{example} [Dirichlet perturbation] \label{eg:Dirichlet}
		Consider the open unit simplex $\Delta^d$ defined by \eqref{eqn:unit.simplex}. We define a commutative operation, called perturbation, on $\Delta^d$ by
		\begin{equation} \label{eqn:perturbation}
			p \oplus q = \left( \frac{p_0q_0}{\sum_{k = 0}^{d} p_k q_k}, \ldots, \frac{p_{d}q_{d}}{\sum_{k = 0}^{d} p_k q_k} \right).
		\end{equation}
		Under the Aitchison geometry in compositional data analysis \cite{A82, EPMB03}, the perturbation is a vector addition. As mentioned in Section \ref{sec:classic.deformation}, the escort transformation plays the role of scalar multiplication.
		
		Fix $\sigma > 0$ to be interpreted as a noise parameter, and let $\lambda = -\sigma < 0$. Let $D = (D_0, D_1, \ldots, D_{d})$ be a Dirichlet random vector with parameters $(\sigma^{-1}/(1 + d), \ldots, \sigma^{-1}/(1 + d))$. Note that as $\sigma \downarrow 0$ the distribution of $D$ concentrates at the barycenter $\overline{e} = (1/(1 + d), \ldots, 1/(1 + d))$ of $\Delta^{d}$. For $p \in \Delta^d$ fixed, consider the random vector $Q = p \oplus D$. It is helpful to think of this probabilistic model as a multiplicative analogue of the additive Gaussian model $Y = \theta + Z$ where $Z \sim N(0, \sigma^2 I)$. In \cite{PW18b} we used this model to construct a probabilistic solution to the Dirichlet transport.
		
		The distribution of $Q$ is given in \cite[Lemma 8]{PW18b}: 
		
		\begin{lemma}
			The density of $Q$ with respect to the Lebesgue measure $dq_1 dq_2 \cdots dq_{d}$ on the domain $\{q_1, \ldots, q_{d} > 0, q_1 + \cdots + q_{d} < 1\}$ is given by
			\begin{equation} \label{eqn:pD.density}
				f_{\lambda}(q \mid p) = \frac{C_{d, \lambda}}{\prod_{i = 0}^{d} q_i} \prod_{i = 0}^{d} \left( \frac{q_i}{p_i} \right)^{\frac{-1}{\lambda (1 + d)}} \left( \sum_{i = 0}^{d} \frac{q_i}{p_i} \right)^{1/\lambda},
			\end{equation}
			where $C_{d, \lambda} > 0$ is a constant depending only on $d$ and $\lambda$.
		\end{lemma}

		Consider on $\Delta^d$ the Dirichlet cost function given by
		\begin{equation} \label{eqn:Dirichlet.cost}
			c(p, q) = \log \left( \frac{1}{1 + d} \sum_{i = 0}^{d} \frac{q_i}{p_i} \right) - \sum_{i = 0}^{d} \frac{1}{1 + d} \log \frac{q_i}{p_i}.
		\end{equation}
		Following \cite{PW18b}, we observe that
		\begin{equation} \label{eqn:f.log.likelihood.as.c}
			f_{\lambda}(q \mid p) \propto \frac{1}{\prod_{i = 0}^{d} q_i}  e^{ -\sigma c(p, q)}.
		\end{equation}
		This justifies the name of the optimal transport problem. We note in passing that $ \frac{1}{\prod_{i = 0}^{d} q_i} dq_1 \cdots dq_{d}$ is a Haar measure on the Aitchison simplex, and in this context it is commonly called the {\it Aitchison measure}.
		
		Now we observe that the distribution of $Q = p \oplus D$ (parameterized by $p$) on $\Delta^d$ can be expressed as a $\lambda$-exponential family after a suitable reparameterization. 
		
		\begin{proposition}[Dirichlet perturbation as an $\lambda$-exponential family] \label{prop:DP.as.F.alpha}
			Let $\lambda = -\sigma < 0$. Consider the state space  $\Delta^d$ be the state space. For $i = 1, \ldots, d$, let $F_i(q) = q_i/q_0$ and $\vartheta_i =  p_0/ (\lambda p_i)$. Note that $\vartheta$ takes values in $\Omega = (-\infty, 0)^{d}$.
			
			Let the reference measure be
			\[
			\mathrm{d}\nu(q) = C_{n, \lambda}' \frac{\prod_{i = 1}^{d} F_i(q)^{-1/\lambda (1 + d)}}{\prod_{i = 0}^{d} q_i} \mathrm{d}q_1 \cdots \mathrm{d}q_{d},
			\]
			where $C_{d, \lambda}' > 0$ is an appropriate constant, and let $\rho( q ; \theta)$ be the density of $Q$ with respect to $\nu$. Then
			\begin{equation} \label{eqn:pD.density.as.F.alpha}
				\rho(q ; \theta) = \left(1 + \lambda \vartheta \cdot F(q) \right)^{1/\lambda} e^{\sum_{i = 1}^{d} \frac{-1}{\lambda (1 + d)} \log (- \vartheta_i)}.
			\end{equation}
			Thus $\rho(\cdot ; \vartheta)$ is a $\lambda$-exponential family with
			\begin{equation} \label{eqn:DT.varphi}
				\varphi_{\lambda}(\vartheta) = \sum_{i = 1}^{d} \frac{1}{\lambda (1 + d)} \log (- \vartheta_i), \quad \vartheta \in (-\infty, 0)^{d}.
			\end{equation} 
		\end{proposition}
		
		We omit the proof as it is mostly direct computation. Statistical applications of the Dirichlet perturbation model and its logarithmic divergence will be investigated in a future paper. See Section \ref{sec:statistical.conseq} for more discussion.
	\end{example}
	
	Our final example in this subsection is the {\it $\alpha$-family} studied by Amari and Nagaoka \cite{AN00}. In Section \ref{sec:alpha.family} we will specialize it to study mixture-type families.
	
	\begin{example} [$\alpha$-family] \label{eg:alpha.family}
		For $\alpha \in \mathbb{R} \setminus \{0\}$ fixed, define the $\alpha$-embedding function
		\begin{equation} \label{eqn:alpha.embedding}
			L_{\alpha}(t) = \frac{2}{1 - \alpha} t^{\frac{1 - \alpha}{2}}, \quad t > 0.
		\end{equation}
		A family $\mathcal{I}$ of positive functions on a state space $\mathcal{X}$ is said to be {\it $\alpha$-affine} if the family $\{ L_{\alpha} \circ f : f \in \mathcal{I}\}$ is convex. A parameterized density $p(\cdot ; \xi)$, $\xi \in \Xi \subset \mathbb{R}^d$, is an {\it $\alpha$-family} if its denormalization 
		\[
		\{ \tau p(\cdot ; \xi) : \tau > 0, \xi \in \Xi\}
		\]
		is $\alpha$-affine (i.e., autoparallel with respect to the $\alpha$-connection). By \cite[(2.76)]{AN00}, an $\alpha$-family has the representation
		\begin{equation} \label{eqn:alpha.family}
			p(x ; \xi) = \left( \sum_{i = 0}^d \theta_i(\xi) F_{i}(x) \right)^{2/(1 - \alpha)}.
		\end{equation}
		See \cite{ZJ04} for a more general $\rho$-affine family and its dual affine geometry.
		
		Let us observe that the $\alpha$-family can be reparameterized as a $\lambda$-exponential family, and hence a $q$-exponential family. To see this, we assume, without loss of generality, that $F_0 > 0$ and $\theta_0(\xi) > 0$. Then, we have
		\begin{equation*}
			\begin{split}
				p(x; \xi) &= \left(\theta_0(\xi) F_0(x) + \sum_{i = 1}^d  \theta_i(\xi) F_{i}(x) \right)^{2/(1 - \alpha)} \\
				&= \left( \theta_0(\xi) F_0(x) \right)^{1/\lambda} \left( 1 + \lambda \sum_{i = 1}^d  \frac{\theta_i(\xi)}{\lambda \theta_0(\xi)} \frac{F_i(x)}{F_0(x)} \right)^{1/\lambda} \\
				&= c(\xi) (1 + \lambda \vartheta(\xi) \cdot \widetilde{F}(x))^{1/\lambda},
			\end{split}
		\end{equation*}
		where $\lambda = \frac{1 - \alpha}{2}$, $\vartheta_i(\xi) = \theta_i(\xi)/ \lambda\theta_0(\xi)$, $\widetilde{F}_i = F_i/F_0$, and $c(\xi)$ is a normalizing constant. Assuming the mapping $\xi \mapsto \vartheta$ is invertible, we see that the $\alpha$-family can be expressed as a $\lambda$-exponential family, where $c(\xi) = e^{-\varphi_{\lambda}(\vartheta(\xi))}$.
	\end{example}

	\section{Mixture-type families under $\lambda$-duality} \label{sec:mixtures}
	The exponential family \eqref{eqn:exponential.family} is not the only family which is naturally associated to the KL-divergence via the Legendre duality. The mixture family plays, in some sense, a dual role \cite[Section 2.3]{a16}. For motivations, we recall that the probability simplex $\Delta^d$ is at the same time an exponential family and a mixture family, and in this case the expectation parameter coincides with the mixture parameter. Nevertheless, the exponential and mixture families are in general distinct objects. In this section we consider mixture-type families under the framework of $\lambda$-duality. We first specialize the $\alpha$-family (Example \ref{eg:alpha.family}) to the mixture case and show that its interpolations are compatible with the geometry of the corresponding $\lambda$-exponential family. Next, we introduce a new {\it $\lambda$-mixture family} which can be regarded as the dual of the $\lambda$-exponential family. We also discuss some relationships between the $\lambda$-exponential and $\lambda$-mixture families.
	
	\subsection{$\alpha$-mixture family} \label{sec:alpha.family}
	Let $p_0(x), p_1(x), \ldots, p_d(x)$ be given (affinely independent) probability densities, with respect to a dominating measure $\nu$, on a given state space. A {\it mixture-type family} may be defined abstractly as a parameterized density $\{p(\cdot ; w)\}_w$, where the mixture parameter $w = (w_0, \ldots, w_d)$ satisfies $w_i \geq 0$, $\sum_i w_i = 1$, such that when $w = e_i = (0, \ldots, 1, \ldots, 0)$ is a vertex of the simplex then $p(\cdot ; w) = p_i$. Geometrically, it is an embedding of the closed simplex $\overline{\Delta^d}$ into the space of all densities, such that the vertices are fixed.
	
	We first consider the $\alpha$-family which can be used to define a mixture-type family. To distinguish it from other families, we call it the {\it $\alpha$-mixture family}. Let $\alpha \neq 1$ be fixed and recall the $\alpha$-embedding function $L_{\alpha}$ defined by \eqref{eqn:alpha.embedding}.
	
	\begin{definition}[$\alpha$-mixture family \cite{A07}] \label{def:alpha.mixture.family}
		The $\alpha$-mixture family with respect to the densities $p_0, \ldots, p_d$ is defined by the $\alpha$-mean, namely
		\begin{equation} \label{eqn:alpha.mixture.family}
			\begin{split}
				p_{\alpha}(x; w) &= c(w) L_{\alpha}^{-1}\left( \sum_{i = 0}^d w_i L_{\alpha}(p_i(x)) \right) \\
				&= c(w) \left( \sum_{i = 0}^d w_i p_i(x)^{(1 - \alpha)/2} \right)^{2/(1 - \alpha)},
			\end{split}
		\end{equation}
		where $c(w)$ is a normalizing constant. Here we assume that the integral which defines $c$ converges.
	\end{definition}
	
	Thus an $\alpha$-mixture family is an $\alpha$-family \eqref{eqn:alpha.family} where each $F_i$ takes the form of a density function $p_i$ and $\theta_0, \ldots, \theta_d$ are the mixture weights. As seen in Example \ref{eg:alpha.family}, the $\alpha$-mixture family can be reparameterized as a $\lambda$-exponential family, and hence a $q$-exponential family, where $\lambda = \frac{1 - \alpha}{2}$ and $q = 1 - \lambda$. In this case there is a more natural way to reparameterize the density. Here we assume that $p_0$ is strictly positive. Let $\vartheta = (\vartheta_1, \ldots, \vartheta_d) = (w_1, \ldots, w_d)$ which takes values in the open convex set $\Omega = \{\vartheta: \vartheta_i > 0, \vartheta_1 + \cdots + \vartheta_d < 1\}$. We have
	\begin{equation} \label{eqn:alpha.mixture.as.lambda}
		\begin{split}
			p_{\alpha}(x; w) &= c(w) \left( \sum_{i = 0}^d w_i p_i(x)^{(1 - \alpha)/2} \right)^{2/(1 - \alpha)} \\
			&= c(w) \left( w_0 (p_0(x))^{\lambda} + \sum_{i = 1}^d w_i (p_i(x))^{\lambda} \right)^{1/\lambda} \\
			&= c(w) p_0(x) \left( 1 + \lambda \sum_{i = 1}^d \vartheta_i \frac{1}{\lambda} \left[ \left( \frac{p_i(x)}{p_0(x)}\right)^{\lambda} - 1 \right] \right)^{1/\lambda} \\
			&= c(w) p_0(x) (1 + \lambda \vartheta \cdot F(x))^{1/\lambda},
		\end{split}
	\end{equation}
	where $F_i(x) = \frac{1}{\lambda} \left[ \left(\frac{p_i(x)}{p_0(x)}\right)^{\lambda} - 1\right]$ is the Box-Cox transformation of the likelihood ratio. So again we get a $\lambda$-exponential family. As a $\lambda$-exponential family, the $\alpha$-mixture family automatically enjoys the properties established in Sections \ref{sec:lambda.exp.family} and \ref{sec:information.geometry} as long as $\lambda < 1$ and the required regularity conditions hold.
	
	
	Amari \cite{A07} showed that the $\alpha$-mixture family can be interpreted as the weighted barycenter with respect to the {\it $\alpha$-divergence}, which we now recall. For $\alpha \neq \pm 1$, the $\alpha$-divergence is defined by
	\begin{equation} \label{eqn:alpha.divergence}
		{\bf D}_{\alpha}[p_1 : p_2] = \frac{4}{1 - \alpha^2} \left(1 - \int p_1^{\frac{1-\alpha}{2}} p_2^{\frac{1 + \alpha}{2}} d\nu \right).
	\end{equation}
	When $\alpha \rightarrow 1$ (respectively $-1$) the $\alpha$-divergence converges to the KL-divergence ${\bf H}(p_1 || p_2)$ (respectively ${\bf H}(p_2 || p_1)$). Note that the $\alpha$-divergence can be expressed as a monotonic transformation of the R\'{e}nyi divergence. Then, by \cite[Theorem 2]{A07}, the $\alpha$-mixture $p_{\alpha}(x; w)$ is the {\it right barycenter} of the densities $p_0, \ldots, p_d$ with respect to the $\alpha$-divergence. Namely, we have
	\begin{equation} \label{eqn:alpha.divergence.barycenter}
		p_{\alpha}(\cdot ; w) = \argmin_p \left\{ \sum_{i = 0}^d w_i {\bf D}_{\alpha}[p_i : p] \right\}.
	\end{equation}
	
	The following result shows that if each density $p_i$ is taken from a base $\lambda$-exponential family $\mathcal{M}$, then the resulting $\alpha$-family (where $\alpha = 1 - 2\lambda$) forms a submanifold (with boundary) of $\mathcal{M}$. We also show that linear interpolations (with respect to $w$) are compatible with the primal coordinates $\vartheta$  of $\mathcal{M}$. This extends and clarifies the analysis of $\alpha$-mixtures between two distributions (called {\it $q$-paths} in \cite{BMBWGSN20}) which adopts subtractive normalization and considers only mixtures of two distributions.

	\begin{proposition} \label{prop:alpha.mix.embed.lambda.exp}
		Consider a $\lambda$-exponential family $\mathcal{M}$ given by
		\[
		p(x; \vartheta) = (1 + \lambda \vartheta \cdot F(x))^{1/\lambda} e^{-\varphi_{\lambda}(\vartheta)},
		\]
		where $1 + \lambda \vartheta \cdot F(x) > 0$ for all $\vartheta$ and $x$. Let $p_i = p(\cdot ; \vartheta^{(i)})$, for $i = 0, 1, \ldots, d$, be $d+1$ members of the family $\mathcal{M}$, each specified by a parameter $\vartheta^{(i)}$.  Let $\alpha = 1 - 2\lambda$. Then the corresponding $\alpha$-mixture family is a subset of $\mathcal{M}$. Specifically,  we have $
		p_{\alpha}(x ; w) = p(x; \vartheta)$, where
		\begin{equation} \label{eqn:alpha.mixture.vartheta}
			\vartheta = \vartheta(w) = \sum_{i = 0}^d \frac{w_i e^{-\lambda \varphi_{\lambda}(\vartheta^{(i)})}}{\sum_j w_j e^{-\lambda \varphi_{\lambda}(\vartheta^{(j)})}} \theta^{(i)}.
		\end{equation}
		In words, given the $d+1$ ``inducing'' elements of $\mathcal{M}$, each $\alpha$-mixture exists and is an element of $\mathcal{M}$. 
	\end{proposition}
	\begin{proof}
		We have
		\begin{equation*}
			\begin{split}
				&\sum_{i = 0}^d w_i p_i^{(1 - \alpha)/2} = \sum_{i = 0}^d w_i p_i^{\lambda} 
				= \sum_{i = 0}^d  w_i (1 + \lambda \vartheta^{(i)} \cdot F) e^{-\lambda \varphi_{\lambda}(\vartheta^{(i)})} \\
				&= c' \left( 1 + \lambda \sum_{i = 0}^d \frac{w_i e^{-\lambda \varphi_{\lambda}(\vartheta^{(i)})}}{\sum_j w_j e^{-\lambda \varphi_{\lambda}(\vartheta^{(j)})}} \theta^{(i)} \cdot F\right) \\
				&= c' (1 + \lambda \vartheta \cdot F),
			\end{split}
		\end{equation*}
		where $c' = c'(w)$ is a constant and $\vartheta$ is given by \eqref{eqn:alpha.mixture.vartheta}. Note that the $(1/\lambda)$-th power of the last expression is integrable since the parameter set of $\mathcal{M}$ is convex by Lemma \ref{lem:Omega.convex}. It follows that $p_{\alpha}(x ; w) = p(x; \vartheta)$.
	\end{proof}
	
	\begin{remark}
		This property is related to the fact that the $\lambda$-exponential family is $\alpha$-affine in the sense of \cite{AN00}. Here, we identified the parameter system $\vartheta$ such that the mapping $w \mapsto \vartheta$ is {\it projective}, i.e., taking straight lines to straight lines. The last property is shown in the following corollary.
	\end{remark}
	
	\begin{corollary} \label{cor:alpha.mixture.curve}
		Under the setting of Proposition \ref{prop:alpha.mix.embed.lambda.exp}, consider a straight line under the mixture parameters $w(t) = (1 - t) w^{(0)} + t w^{(1)}$ of the $\alpha$-family. Then the corresponding probability density functions $\{p_{\alpha}(\cdot ; w(t)) : 0 \leq t \leq 1\}$, which are all members of the $\lambda$-exponential family $\mathcal{M}$, trace out a straight line under the primal coordinate system $\vartheta$.
	\end{corollary}
	\begin{proof}
		Let $\vartheta(t)$, given by \eqref{eqn:alpha.mixture.vartheta}, be the $\vartheta$-coordinate of $p_{\alpha}(\cdot; w(t))$ as an element of $\mathcal{M}$. We have
		\begin{equation*}
			\begin{split}
				\vartheta(t) = \sum_{i = 0}^d \frac{((1 - t) w_i^{(0)} + t w_i^{(1)})a_i }{\sum_j ((1 - t) w_j^{(0)} + t w_j^{(1)}) a_j} \theta^{(i)},
			\end{split}
		\end{equation*}
		where $a_i = e^{-\lambda \varphi_{\lambda}(\theta^{(i)})}$.
		
		Consider the weight of the $i$-term. Write
		\begin{equation} \label{eqn:weight.analysis}
			\begin{split}
				&\frac{((1 - t) w_i^{(0)} + t w_i^{(1)})a_i }{\sum_j ((1 - t) w_j^{(0)} + t w_j^{(1)}) a_j} \\
				&= \frac{(1 - t) \sum_j w_j^{(0)} a_j}{\sum_j [ \cdots ] a_j} \frac{w_i^{(0)} a_i}{\sum_j w_j^{(0)} a_j} + \frac{t \sum_j w_j^{(1)} a_j}{\sum_j [ \cdots ] a_j} \frac{w_i^{(1)} a_i}{\sum_j w_j^{(0)} a_j} \\
				&= (1 - s(t)) \frac{w_i^{(0)} a_i}{\sum_j w_j^{(0)} a_j} + s(t) \frac{w_i^{(1)} a_i}{\sum_j w_j^{(0)} a_j},
			\end{split}
		\end{equation}
		where $[\cdots] = (1 - t) w_j^{(0)} + t w_j^{(1)}$. Note that $s(t)$ does not depend on $i$ and increases from $0$ to $1$. Thus we have
		\[
		\vartheta(t) = (1 - s(t)) \vartheta(0) + s(t) \vartheta(1),
		\]
		which is a time change of a constant-speed straight line under the $\vartheta$ coordinates. 
	\end{proof}
	
	According to the geometric language explained in Section \ref{sec:information.geometry} (also see Remark \ref{rmk:primal.dual.coordinates}), the path $\{p_{\alpha}(\cdot ; w(t)) : 0 \leq t \leq 1\}$ is a {\it primal pre-geodesic} under the information geometry induced by the $\lambda$-logarithmic divergence ${\bf L}_{\lambda, \varphi_{\lambda}}$ on $\mathcal{M}$. Indeed, equations analogous to \eqref{eqn:weight.analysis} also appear naturally in \cite{PW18,W18}, and the time change is related to the fact that the underlying geometry is {\it projectively flat}.  
	
	\begin{figure}[!t]
		\centering
		\includegraphics[width=4in]{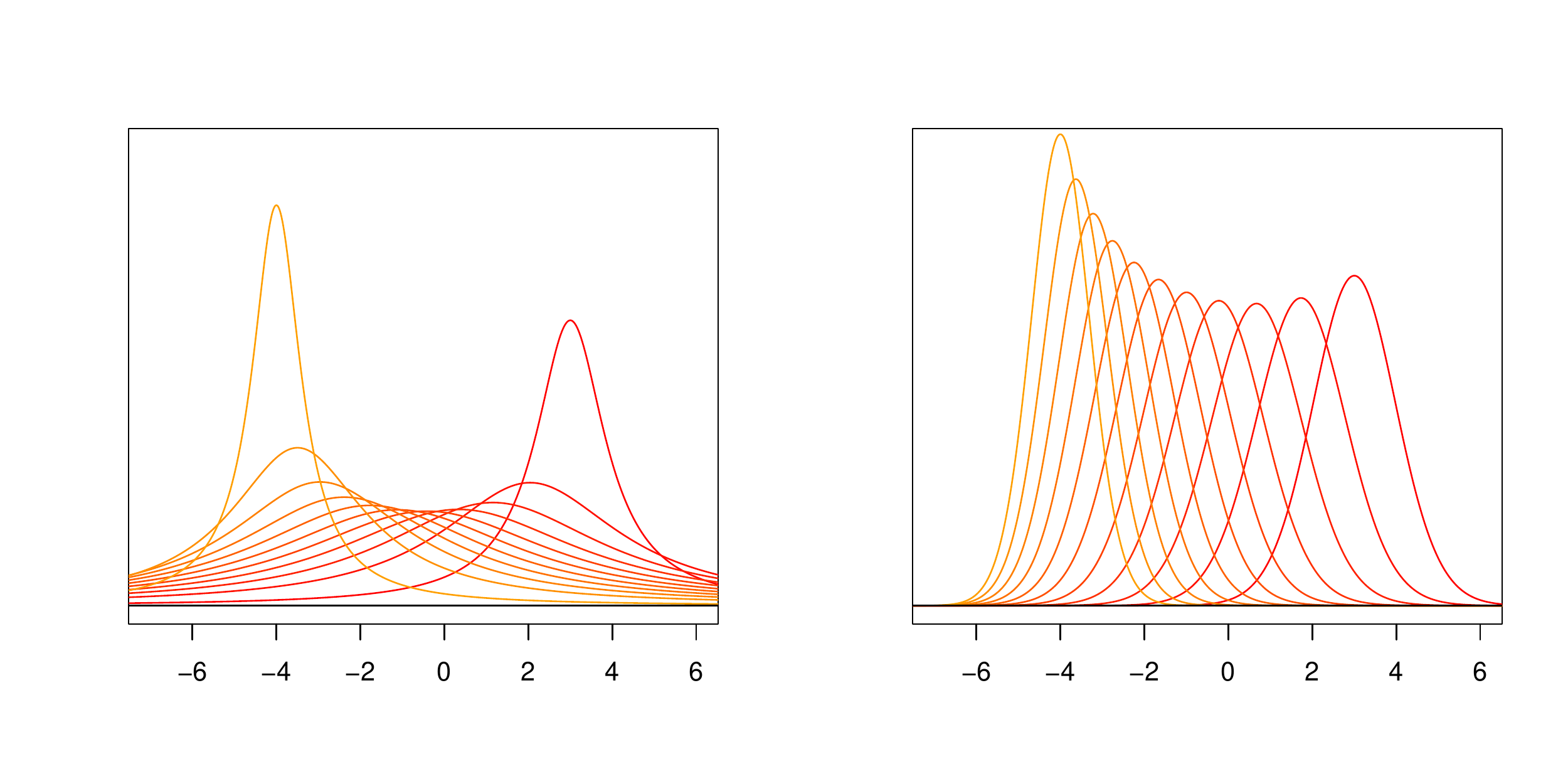}
		\vspace{-0.7cm}
		\caption{The $\alpha$-family formed by two $t$-distributions (with different location and scale parameters) on the real line (see Example \ref{eg:t.dist}). Left: df (degree of freedom) $= 3$, so that $\alpha = 0.5$. Right: $\text{df} = 30$, so that $\alpha \approx 1.13$. In both bases we use the location and scale parameters $(\mu_0, \sigma_0) = (-4, 0.7)$, $(\mu_1, \sigma_1) = (3, 1)$ for $p_0$ and $p_1$. The interpolation corresponds to a straight line under the $\vartheta$-coordinate system of the associated $\lambda$-exponential family (see Corollary \ref{cor:alpha.mixture.curve}), but is a nonlinear curve under the $\theta$-coordinate system of the corresponding $q$-exponential family. As $\text{df} \rightarrow \infty$ (so $\alpha \rightarrow 1$) the $\alpha$-family converges to the {\it exponential mixture} which is analogous to (but different from) McCann's displacement interpolation \cite{M97} between two Gaussian distributions.}
		\label{fig:t.mixture}
	\end{figure}
	
	\begin{example}[Multivariate $t$-distribution] \label{eg:t.dist}
		Let the state space be $\mathbb{R}^n$. The multivariate $t$-distribution with $k$ degrees of freedom, and with parameters $(\mu, \Sigma)$ (where $\mu \in \mathbb{R}^n$ and $\Sigma$ is a $n\times n$ positive definite matrix) has density (with respect to the Lebesgue measure) given by
		\begin{equation} \label{eqn:MV.t.distribution}
			p(x ; k, \mu, \Sigma) = c \left[1 + \frac{1}{k} (x - \mu)^{\top} \Sigma^{-1} (x - \mu) \right]^{-(k + n)/2},
		\end{equation}
		where $c = c(\mu, \Sigma) > 0$ is a normalization constant. It can be shown that for $k$ fixed this can be reparameterized as a $\lambda$-exponential family where $\lambda = \frac{-2}{k + n} < 0$ (this generalizes Examples \ref{ex:Cauchy} and \ref{ex:q.Guassian}). Thus we may write
		\begin{equation} \label{eqn:t.as.q.exp}
			p(x; k, \mu, \Sigma) = p(x ; \vartheta) = (1 + \lambda \vartheta \cdot F(x))^{\frac{1}{\lambda}} e^{-\varphi_{\lambda}(\vartheta)},
		\end{equation}
		where $F(x) = ( (x_i)_i, (x_ix_j)_{i \leq j})$. Note that as $k \rightarrow \infty$ it converges to the Gaussian distribution $N(\mu, \Sigma)$. 
		
		Let $p_i = p(\cdot ; k, \mu^{(i)}, \Sigma^{(i)})$ for $i = 0, 1, \ldots, d$ and consider the $\alpha$-mixture family where $\alpha = 1 - 2\lambda$. By Proposition \ref{prop:alpha.mix.embed.lambda.exp}, each $\alpha$-mixture belongs to the same family, and hence is a $t$-distribution. This fact may appear strange at first sight as the ordinary mixture (of unimodal distributions) is typically multimodal. However, as mentioned above, the $\alpha$-mixture \eqref{eqn:alpha.family} may also be interpreted as the barycenter with respect to the $\alpha$-divergence \cite{A07}. Here, the $\alpha$-mixture is analogous to  the Wasserstein barycenter in optimal transport theory \cite{AC11}, where it can be shown that the Wasserstein barycenter of Gaussian distributions is also Gaussian. In Figure \ref{fig:t.mixture} (inspired by \cite{BMBWGSN20}) we give a graphical illustration of the $\alpha$-mixture of two $t$-distributions on the real line. We remark that here $\alpha$ is chosen as a function of the degrees of freedom $k$. 
	\end{example}

	\subsection{$\lambda$-mixture family} \label{sec:lambda.mixture}
	In \eqref{eqn:alpha.mixture.as.lambda}, where we express the $\alpha$-mixture family as a $\lambda$-exponential family, the mixture parameter $w$ plays the role of the primal variable $\vartheta$. However, as mentioned in Section \ref{sec:Legendre.duality}, the negative Shannon entropy of a (conventional) mixture family is convex in the mixture parameter which plays the role of the dual variable, and its Bregman divergence is the KL-divergence. In this subsection we introduce a new $\lambda$-mixture family which preserves this analogy. To motivate the definition, recall Example \ref{eg:Renyi.simplex} where the negative R\'{e}nyi entropy of the escort distribution is $c_{\lambda}$-convex on the simplex. Recall also that $\mathcal{E}_q[\cdot]$ is the escort transformation with exponent $q$ (see \eqref{eqn:escort.transformation.def}):
	
	\begin{definition} [$\lambda$-mixture family] \label{def:lambda.mixture}
		Let $\lambda \neq 1$ be given and let $q = 1 - \lambda$. We define the $\lambda$-mixture family with respect to the densities $p_0, \ldots, p_d$ by  
		\begin{equation} \label{eqn:lambda.mixture.density}
			p_{\eta}(x) = p(x ; \eta) = \mathcal{E}_{1/q}\left[ \sum_{i = 0}^d \eta_i \mathcal{E}_{q} [p_i] \right],
		\end{equation}
		where $\eta_i \geq 0$, $\sum_{i = 0}^d \eta_i = 1$ are the $\lambda$-mixture parameters, provided that the integrals involved all converge. Explicitly, we have
		\begin{equation} \label{eqn:lambda.mixture.density2}
		p_{\eta}(x) = \frac{1}{\int (\sum_{j = 0}^d \eta_j \widetilde{p}_j )^{1/q} \mathrm{d} \nu} \left( \sum_{i = 0}^d \eta_i \widetilde{p}_i \right)^{1/q},
	\end{equation}
		where $\widetilde{p}_i = \mathcal{E}_q[p_i] = p_i^q / \int p_i^q \mathrm{d}\nu$. 
	\end{definition}
	
	From \eqref{eqn:lambda.mixture.density}, if $\widetilde{p}_{\eta} = \mathcal{E}_q[p_{\eta}]$ is the escort transformation of $p_{\eta}$, then
	\[
	\widetilde{p}_{\eta} = \sum_{i = 0}^d \eta_i \widetilde{p}_i
	\]
	is an ordinary mixture family with respect to the escort distributions $\widetilde{p}_0, \ldots, \widetilde{p}_d$. Thus the $\lambda$-mixture family is nothing but the escort transformation (with exponent $1/q$) of an ordinary mixture family. See Section \ref{sec:further.relations} for the relationship between the $\lambda$-mixture and $\alpha$-mixture families.
	
	Before proceeding we illustrate the $\lambda$-mixture family with a concrete example. 
	
	\begin{example} \label{eg:mixture}
		Consider the finite state space $\mathcal{X} = \{0, 1, 2\}$ and let $\nu$ be the counting measure. Consider for the sake of illustration the three densities
		\begin{equation*}
			\begin{split}
				p_0 &= (0.8, 0.1, 0.1), \\
				p_1 &= (0.02, 0.9, 0.08), \\
				p_2 &= (0.35, 0.2, 0.45).
			\end{split}
		\end{equation*}
		In Figure \ref{fig:mixture} we visualize the corresponding $\lambda$-mixture family where $\lambda = -2$ (left) and $\lambda = 0.7$ (right). Each density $p(\cdot ; \eta)$, where $\eta$ ranges over a uniform grid on the unit simplex, is shown as a blue dot on the unit simplex $\Delta^2$.
		
		\begin{figure}[!t]
			\centering
			\includegraphics[width=1.73in]{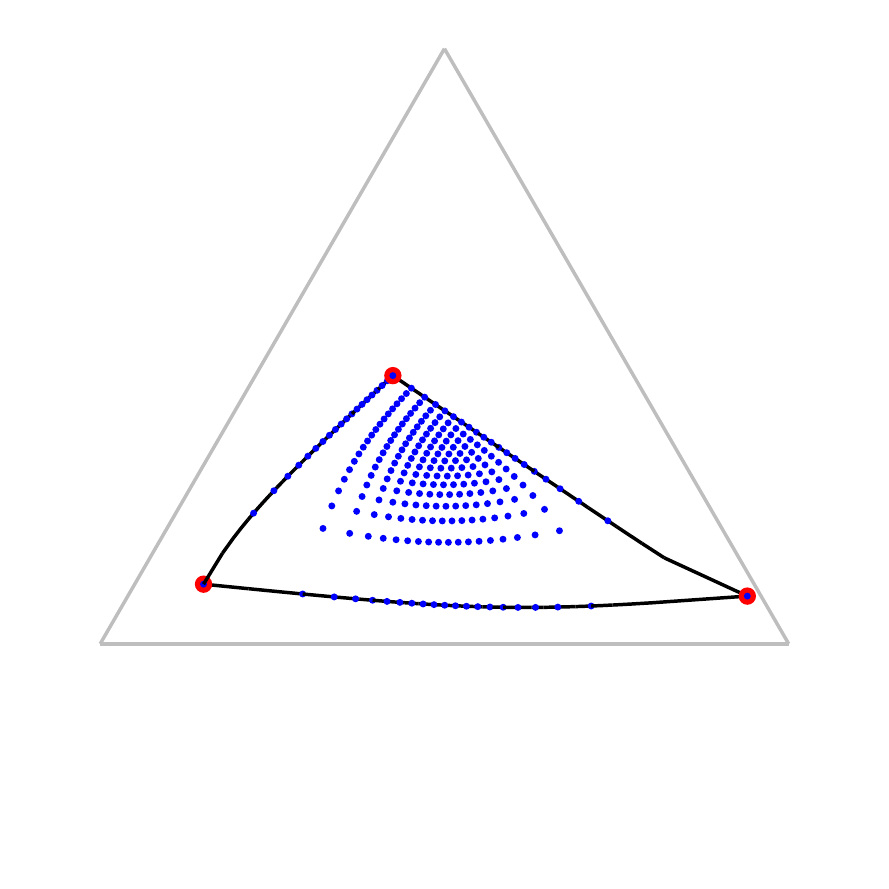}
			\hspace{-0.2cm}
			\includegraphics[width=1.73in]{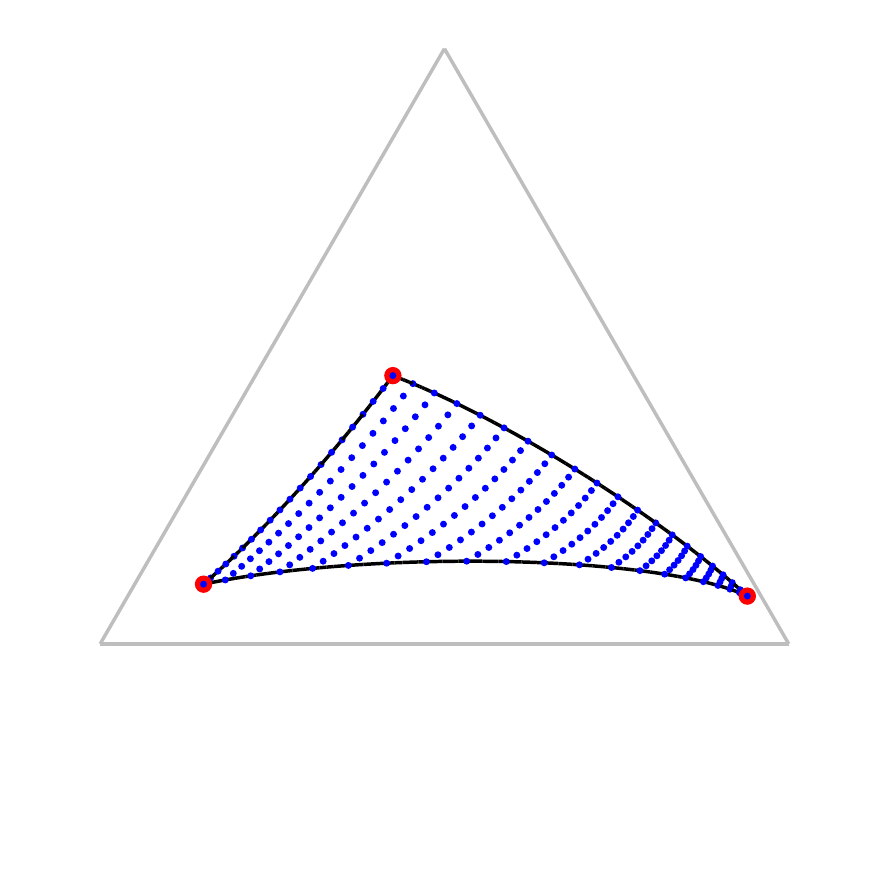}
			\vspace{-1.2cm}
			\caption{Graphical illustration of the $\lambda$-mixture family in the context of Example \ref{eg:mixture}. Left: $\lambda = -2$. Right: $\lambda = 0.7$. The red points show the densities $p_i$, $i = 0, 1, 2$. Each $p(\cdot ; \eta)$, where $\eta$ ranges over a uniform grid on the closed simplex $\overline{\Delta^2}$, is shown as a blue dot in $\Delta^2$ (where the superscript denotes the dimension of the simplex $\Delta$.}
			\label{fig:mixture}
		\end{figure}
		
		Observe that the $\lambda$-mixture family leads to nonlinear interpolating curves on the simplex. For example, when $\lambda < 0$ the blue dots are denser when $\eta$ is near the center of the simplex, and are far apart when $\eta$ is close to the boundary. When $\lambda \rightarrow 0$ we cover the ordinary mixture family \eqref{eqn:mixture.family}.
	\end{example}
	
	The following result shows that the $\lambda$-mixture family is naturally compatible with the $\lambda$-duality, and preserves the analogy with the mixture family (see \eqref{eqn:mixture.family.Bregman}).
	
	\begin{theorem} \label{thm:lambda.mixture}
		Consider a $\lambda$-mixture family where $\lambda < 1$ (or $q = 1 - \lambda > 0$) (this assumes implicitly that all integrals involved in \eqref{eqn:lambda.mixture.density} converge for all $\eta$), such that the integral $\int (\sum_{i = 0}^d \eta_i \widetilde{p}_i )^{1/q} \mathrm{d}\nu$ can be differentiated with respect to $\eta$ under the integral sign when $\eta \in \Delta^d$. Consider the negative R\'{e}nyi entropy
		\begin{equation} \label{eqn:lambda.mixture.dual.potential}
			\psi(\eta) = - {\bf H}^{\mathrm{R\acute{e}nyi}}_q ( p_{\eta}), \quad q = 1 - \lambda.
		\end{equation}
		Then $\frac{1}{\lambda} (e^{\lambda \psi} - 1)$ is convex. Moreover, the $\lambda$-logarithmic divergence of $\psi$ is the R\'{e}nyi divergence with order $q$: 
		\begin{equation} \label{eqn:lambda.mixture.L.div}
			{\bf L}_{\lambda, \psi}[\eta : \eta'] = {\bf H}^{\mathrm{R\acute{e}nyi}}_q ( p_{\eta} || p_{\eta'}), \quad \eta' \in \Delta^d.
		\end{equation}
	\end{theorem}
	\begin{proof}
		Let $Z(\eta) = \int (\sum_j \eta_j \widetilde{p}_j)^{1/q} \mathrm{d}\nu$ be the partition function in \eqref{eqn:lambda.mixture.density2}, so that
		\[
		p(x ; \eta) = \frac{1}{Z(\eta)} \left( \sum_{j = 0}^d \eta_j \widetilde{p}_j \right)^{1/q}.
		\]
		By \eqref{eqn:lambda.mixture.density} and the definition of R\'{e}nyi entropy  \eqref{eqn:Renyi.entropy}, we have
		\begin{equation} \label{eqn:psi.expression}
			\psi(\eta) = \frac{q}{1 - q} \log \int \left( \sum_{i = 0}^d \eta_i \widetilde{p}_i \right)^{1/q} \mathrm{d}\nu = \frac{q}{1 - q} \log Z(\eta).
		\end{equation}
		
		Let $\eta, \eta'$ be given. By a direct computation, we have
		\begin{equation*}
			\begin{split}
				1 + \lambda \nabla \psi(\eta') \cdot (\eta - \eta')
				&= \frac{ \int (\sum_i \eta_i' \widetilde{p}_i)^{1/q - 1} (\sum_i \eta_i \widetilde{p}_i) \mathrm{d}\nu}{\int (\sum_i \eta'_i \widetilde{p}_i )^{1/q} \mathrm{d}\nu}\\
				&= \frac{Z(\eta)^q}{Z(\eta')^{q}} \int p_{\eta}(x)^q p_{\eta'}(x)^{1 - q}  \mathrm{d}\nu.\\
			\end{split}
		\end{equation*}
		It follows that
		\begin{equation*}
			\begin{split}
				&\psi(\eta) - \psi(\eta') - \frac{1}{\lambda} \log (1 + \lambda \nabla \psi(\eta') \cdot (\eta - \eta'))\\ 
				&= \frac{1}{q - 1} \log \int p_{\eta}(x)^q p_{\eta'}(x)^{1 - q} \mathrm{d}\nu(x) \\
				&= {\bf H}^{\mathrm{R\acute{e}nyi}}_q ( p_{\eta} || p_{\eta'} ) \geq 0.
			\end{split}
		\end{equation*}
		This argument shows that $\frac{1}{\lambda} (e^{\lambda \psi} - 1)$ is convex whenever $q = 1 - \lambda > 0$, and the $\lambda$-logarithmic divergence of $\psi$ is the R\'{e}nyi divergence of order $q$.
	\end{proof}
	
	\begin{remark}
		It is easy to verify by examples that Theorem \ref{thm:lambda.mixture} does not hold for the $\alpha$-mixture family if we define instead $\psi(w) = - {\bf H}^{\mathrm{R\acute{e}nyi}}_q ( p_{\alpha}(\cdot ; w))$. For an $\alpha$-mixture family (which can be expressed as a $\lambda$-exponential family), it is the $c_{\lambda}$-conjugate $\psi_{\lambda}$ of the divisive $\lambda$-potential $\varphi_{\lambda}$ which can be expressed as a R\'{e}nyi entropy (see Theorem \ref{thm:Renyi.divergence}).
	\end{remark}

	
	Now we place Example \ref{eg:Renyi.simplex}  in the proper context.
	
	\begin{example} [Finite simplex] \label{eg:lambda.mixture.simplex}
		Consider the unit simplex $\Delta^d$ as in \eqref{eqn:unit.simplex}. Each $u \in \Delta^d$ can be regarded as the density $p$ of a probability measure on a finite set $\mathcal{X} = \{0, 1, \ldots, d\}$ with respect to the counting measure. So we may view $\Delta^d \cong \mathcal{P}_+(\mathcal{X})$ as the set of strictly positive probability measures on $\mathcal{X}$. We write $u_x = p(x)$, $x \in \mathcal{X}$.
		
		For $\lambda \neq 0$, $\Delta^d$ can be expressed as a $\lambda$-exponential family. To see this, let $F_i(x) = \delta_i(x)$ be the indicator of $i \in \mathcal{X}$, $i = 1, \ldots, d$. Then we may write
		\begin{equation} \label{eqn:simplex.as.lambda.exp}
			u_x = p_{\vartheta}(x) = \left(1 + \lambda \sum_{i = 1}^d \vartheta_i \delta_i(x)\right)^{1/\lambda} e^{-\varphi_{\lambda}(\vartheta)},
		\end{equation}
		where
		\begin{equation} \label{eqn:simplex.vartheta}
			\vartheta_i = \frac{1}{\lambda} \left[ \left( \frac{u_i}{u_0} \right)^{\lambda} - 1\right]
		\end{equation}
		is the Box-Cox transformation of $u_i/u_0$, and the parameter set is $\Omega = \{\vartheta \in \mathbb{R}^d: 1 + \lambda \vartheta^i > 0 \ \forall i \}$. Note that when $\lambda \rightarrow 0$, $\vartheta$ reduces to the usual exponential coordinates $\theta_i = \log (u_i/u_0)$. The potential function, given by
		\[
		\varphi_{\lambda}(\vartheta) = -\log u_0 = \log \left( 1 + \sum_{i= 1}^d (1 + \lambda \vartheta_i)^{1/\lambda}\right),
		\]
		is $c_{\lambda}$-convex for $\lambda < 1$. By Theorem \ref{thm:c.exp.dual.variable}, the dual variable $\eta = \nabla^{c_{\lambda}} \varphi_{\lambda} (\vartheta)$ gives the escort distribution with exponent $q$:
		\begin{equation} \label{eqn:escort.probability}
			\eta_i = \frac{(p_{\vartheta}(i))^q}{\sum_{j = 0}^d  (p_{\vartheta}(j))^q }, \quad i = 1, \ldots, d.
		\end{equation}
		We define $\eta_0 = 1 - \sum_{i = 1}^d \eta_i$. Dually, for $i = 0, 1, \ldots, d$, let $p_i(x) = \delta_i(x)$ be the density of the point mass at point $i$ and note that $\widetilde{p}_i = \mathcal{E}_q[p_i] = p_i$. Consider the corresponding $\lambda$-mixture family $p_{\eta}$ (not to be confused with $p_{\vartheta}$). By \eqref{eqn:lambda.mixture.density}, we have
		\begin{equation} \label{eqn:simplex.as.lambda.mixture}
			p_{\eta}(i) = \frac{\eta_i^{1/q}}{\sum_{j = 0}^d \eta_j^{1/q}}.
		\end{equation}
		Note that if $\eta$ is given by \eqref{eqn:escort.probability}, then we have $p_{\vartheta} = p_{\eta}$. Thus the simplex is at the same time a $\lambda$-exponential family and a $\lambda$-mixture family, and the mixture parameter of the $\lambda$-mixture family is the dual parameter of the $\lambda$-exponential family. That the unit simplex can be expressed as a generalized exponential family (including the $q$-exponential family) was observed in earlier works such as \cite{AOM12}. Here, we use the $\lambda$-representation and connect it with the $\lambda$-mixture family. 
	\end{example}
	
	Formula \eqref{eqn:simplex.vartheta}, which gives the ``primal'' variable (if we regard $\eta$ as the dual one), can be extended to a general $\lambda$-mixture family. Since the set $\{\eta \in (0, 1)^{1 + d}  : \eta_0 + \cdots + \eta_d = 1\}$ is not open, we consider instead the open domain
	\[
	\Omega' = \{ \bar{\eta} = (\bar{\eta}_1, \ldots, \bar{\eta}_d) \in (0, 1)^d : \bar{\eta}_1 + \cdots + \bar{\eta}_d < 1\},
	\]
	where the notation $\bar{\eta}$ signifies that the $0$-th coordinate is dropped. Consider the negative R\'{e}nyi entropy $\psi$, given by \eqref{eqn:lambda.mixture.dual.potential}, as a function of $\bar{\eta}$. Note that 
	\begin{equation} \label{eqn:derivative.eta.bar}
		\frac{\partial \psi}{\partial \bar{\eta}_i} = \frac{\partial \psi}{\partial \eta_i} - \frac{\partial \psi}{\partial \eta_0}, \quad i = 1, \ldots, d.
	\end{equation}
	Note that $\frac{1}{\lambda} (e^{\lambda \psi} - 1)$ is convex in $\bar{\eta}$ (which is a linear transformation of $\eta$) and so we may consider the $\lambda$-duality. One can verify that in the context of Example \ref{eg:lambda.mixture.simplex} the formula \eqref{eqn:mixture.family.theta} below reduces to \eqref{eqn:simplex.vartheta}.
	
	\begin{proposition} \label{prop:lambda.mixture.primal.variable}
		With the above notations, consider the primal variable $\bar{\vartheta} = \nabla_{\bar{\eta}}^{c_{\lambda}} \psi(\bar{\eta})$. We have
		\begin{equation} \label{eqn:mixture.family.theta}
			\bar{\vartheta}^i = \frac{1}{\lambda} \left[ \frac{\int p(x ; \eta)^{\lambda} \widetilde{p}_i(x) d\nu}{\int p(x ; \eta)^{\lambda} \widetilde{p}_0(x) d\nu} - 1\right], \quad i = 1, \ldots, d.
		\end{equation}
	\end{proposition}
	\begin{proof}
		Given $\bar{\eta} \in \Omega$, consider $\eta$ where $\eta_i = \bar{\eta}_i$ for $1 \leq i \leq d$ and $\eta_0 = 1 - \sum_i \bar{\eta}_i$. Using \eqref{eqn:derivative.eta.bar}, we have
		\[
		\frac{\partial \psi}{\partial \bar{\eta}_i} = \frac{1}{\lambda Z(\eta)} \int \left( \sum_{j = 0}^d \eta_j \widetilde{p}_j\right)^{1/q - 1} (p_i - p_0) d\nu.
		\]
		Now \eqref{eqn:mixture.family.theta} follows from the definition of the $\lambda$-deformed Legendre transformation.
	\end{proof}
	
	\subsection{Further relations} \label{sec:further.relations}
	We have seen in \eqref{eqn:alpha.mixture.as.lambda} that a $\alpha$-mixture family can be reparameterized as a $\lambda$-exponential family with $\lambda = \frac{1 - \alpha}{2}$. In this subsection we consider other relations between the $\alpha$-mixture family, $\lambda$-mixture family and $\lambda$-exponential family.
	
	The $\alpha$-mixture family and the $\lambda$-mixture family differ in the way the densities are normalized before taking the average: in \eqref{eqn:alpha.family} one considers $\sum_i w_i p_i^{(1 - \alpha)/2}$, while in \eqref{eqn:lambda.mixture.density} we have instead $\sum_i \eta_i \mathcal{E}_q[p_i]$. The $\alpha$-mixture family is more general since the first sum always exists while the escort distributions $\mathcal{E}_q[p_i]$ may not be well-defined (since the integral $\int p_i^q \mathrm{d}\nu$ may diverge). Nevertheless, when the state space is a finite or bounded and $p_i > 0$ for all $i$, the $\lambda$-mixture family is always well-defined. Again we note that the main motivation for the $\lambda$-mixture family is Theorem \ref{thm:lambda.mixture} which is analogous to the Legendre duality of the ordinary mixture family. 
	
	Now we show that a $\lambda$-mixture family, when exists, is a reparameterization of the $\alpha$-mixture family. So the two families give the same collection of distributions.
	
	\begin{proposition}
		A $\lambda$-mixture family is a reparameterization of the $\alpha$-mixture family (with the same base densities $p_0, \ldots, p_d$) where $\lambda = \frac{\alpha + 1}{2}$.
	\end{proposition}
	\begin{proof}
		Consider a $\lambda$-mixture family $p(\cdot ; \eta)$ of $p_0, \ldots, p_d$. Let $q = 1 - \lambda$. Write $\mathcal{E}_q[p_i] = p_i^q / Z_i$, where $Z_i = \int p_i^q d\nu$ is a constant which is finite by assumption. In what follows $c, c'$ are positive constants. Then we have
		\begin{equation*}
			\begin{split}
				p(\cdot ; \eta) &= c \left( \sum_{i = 0}^d \eta_i \frac{p_i^q}{Z_i} \right)^{1/q} 
				= c' \left( \sum_{i = 0}^d \frac{\eta_i/Z_i}{\sum_j \eta_j / Z_j} p_i^q \right)^{1/q} \\
				&= c' \left( \sum_{i = 0}^d w_i p_i^{(1 - \alpha)/2} \right)^{2/(1 - \alpha)} = p_{\alpha}(\cdot ; w),
			\end{split}
		\end{equation*}
		where the two sets of mixture parameters are related by $w_i = \frac{\eta_i/Z_i}{\sum_j \eta_j / Z_j}$ and $\eta_i = \frac{w_i Z_i}{\sum_j w_j Z_j}$. Here the parameter transformations are also projective.
	\end{proof}
	
	Because of this result, the $\alpha$-mixture family and $\lambda$-mixture family (when $\lambda = \frac{\alpha + 1}{2}$) are qualitatively similar. It is interesting that the reparameterization $w \mapsto \eta$ enables us to link the family with the $\lambda$-duality. In view of this result  we have the following
	
	\begin{corollary}
		Suppose $p_0 > 0$. Then a $\lambda$-mixture family is, after a reparameterization, a $\lambda'$-exponential family, and hence a $q'$-exponential family, where $\lambda' = 1 - \lambda$ and $q' = 1 -\lambda' = \lambda$.
	\end{corollary}
	
	Thus both the $\alpha$-mixture family and the $\lambda$-mixture family can be regarded as special cases of the $\lambda$-exponential family (for different values of $\lambda$). This is a special feature of our framework as it is well-known that in general a mixture family cannot be expressed as an exponential family. We also remark that while the $\lambda$-exponential family and the $q$-exponential family give algebraically the most general expressions, the $\lambda$-mixture family satisfies the duality in Theorem \ref{thm:lambda.mixture} (which adopts the R\'{e}nyi entropy) when $\lambda < 1$ (or $q > 0$), while the results of $\lambda'$-exponential family (which adopts another potential function) requires $\lambda' < 1$.

	
	\section{Information geometry of $\lambda$-logarithmic divergence} \label{sec:information.geometry}
	Both the $\lambda$-exponential and $\lambda$-mixture families can be regarded as manifolds of probability distributions. By Theorem \ref{thm:Renyi.divergence} and Theorem \ref{thm:lambda.mixture}, each family is associated to a $\lambda$-logarithmic divergence (Definition \ref{def:L.div}) which is shown to be a R\'{e}nyi divergence.  Again, this parallels the classical exponential and mixture families which are connected to the Kullback-Leibler divergence, see \eqref{eqn:exp.family.Bregman} and \eqref{eqn:mixture.family.Bregman}. In information geometry \cite{a16}, a divergence induces a geometry on the underlying manifold of probability families that carries a dualistic structure. Formally, it consists of a Riemannian metric as well as a pair of mutually dual and torsion-free affine connections; in differential geometry, this structure is also known as a ``statistical manifold''. Using the results of \cite{W18}, in this section we summarize the information geometry induced by a $\lambda$-logarithmic divergence. 
	
	Let $\Omega$ be an open convex set of $\mathbb{R}^d$ and let ${\bf L}_{\lambda, \varphi}[\cdot : \cdot]$ be a $\lambda$-logarithmic divergence on $\Omega$. For concreteness, the reader may keep in mind the situation of Theorem \ref{thm:Renyi.divergence} where $\Omega$ is the natural parameter set of a $\lambda$-exponential family and $\varphi = \varphi_{\lambda}$ is the divisive potential function. 
	
	\subsection{Riemannian metric}
	Consider two nearly points $\vartheta$ and $\vartheta + \Delta \vartheta$ in $\Omega$ and consider the divergence ${\bf L}_{\lambda, \varphi}[ \vartheta + \Delta \vartheta : \vartheta]$. Applying a Taylor approximation to \eqref{eqn:lambda.logdivergence}, we have the local quadratic approximation
	\begin{equation} \label{eqn:L.div.quadratic.approx}
		{\bf L}_{\lambda, \varphi}[ \vartheta + \Delta \vartheta : \vartheta] = \frac{1}{2} (\Delta \vartheta)^{\top} g(\vartheta) (\Delta \vartheta) + O(|\Delta \vartheta|^3),
	\end{equation}
	where the matrix 
	\begin{equation} \label{eqn:L.div.metric}
		g(\vartheta) = \nabla^2 \varphi(\vartheta) + \lambda (\nabla \varphi(\vartheta))(\nabla \varphi(\vartheta))^{\top}
	\end{equation}
	is positive definite as can be seen from \eqref{eqn:second.derivative.condition}. We regard $g(\vartheta)$ as a metric tensor on the manifold $\Omega$:
	\[
	\mathrm{d} s^2 = \sum_{i, j} g_{ij}(\vartheta) \mathrm{d} \vartheta^i \mathrm{d} \vartheta^j.
	\]
	
	Let us compare \eqref{eqn:L.div.metric} with the metric induced by a Bregman divergence. If $\phi(\theta)$ is a convex potential which defines a Bregman divergence, the corresponding metric tensor is the Hessian $\nabla^2 \phi(\theta)$ \cite[Chapter 1]{a16}. From \eqref{eqn:L.div.metric}, we see that when $\varphi$ is itself convex, then the metric induced by the $\lambda$-logarithmic divergence is a rank-$1$ correction of a Hessian metric. Alternatively, from \eqref{eqn:second.derivative.condition}, we may also cast \eqref{eqn:L.div.metric} as 
	\[
	g(\vartheta) = e^{-\lambda \varphi(\vartheta)} \nabla^2 \Phi(\vartheta),
	\]
	where $\Phi = \frac{1}{\lambda} (e^{\lambda \varphi} - 1)$ is convex. Thus, the metric can also be regarded as a conformal transformation of the Hessian metric induced by $\Phi$; hence, we call it a {\it conformal Hessian metric}. In fact, as shown in \cite{WY19}, the $\lambda$-logarithmic divergence may be regarded as a monotone transformation of a conformal Bregman divergence.

	For the $\lambda$-exponential and $\lambda$-mixture families, where the divergence is a R\'{e}nyi divergence, the induced Riemannian metric is $q$ times the Fisher information metric (see \cite[Section III-H]{VH14}). This is different from classical deformation theory (see for example \cite{AO11}) where the induced metric is a conformal transformation of the Fisher metric. By using the $\lambda$-duality we recover exactly the Fisher metric up to a multiplicative constant.
	
	\begin{remark}
		In \cite{NZ18}, it was shown that deformed exponential families (under subtractive normalization) admit in general both a Hessian metric and a conformal Hessian metric, where the Fisher metric is deformed as a $(\rho, \tau)$-metric. 
	\end{remark}
	
	\subsection{Geodesics and generalized Pythagorean theorem}
	Remarkably, the $\lambda$-logarithmic divergence (which contains the Bregman divergence in the limit) satisfies a {\it generalized Pythagorean theorem}. To state the result it we need the notion of primal and dual geodesics. We regard $\vartheta \in \Omega$ as the primal coordinate system, and $\eta = \nabla^{c_{\lambda}} \varphi(\vartheta)$ as the dual coordinate system. Let $\gamma(t)$ be a curve in the statistical manifold. 
	\begin{itemize}
		\item We say that $\gamma$ is a {\it primal pre-geodesic} if its image under the primal coordinate system is a straight line.
		\item Similarly, $\gamma$ is a {\it dual pre-geodesic} if its image is a straight line under the dual coordinate system.
	\end{itemize}
	Note that we call the curves pre-geodesics because the actual geodesics (defined by the primal and dual geodesic equations) run in non-constant speed in the respective coordinate systems. We refer the reader to \cite{W18} for the associated primal and dual affine connections which define the geodesic equations. Nevertheless, the trajectories of these geodesics are straight lines under the respective coordinate systems. Because of this feature, the geometry is said to be {\it dually projectively flat}. We are now ready to state the theorem.
	
	\begin{theorem} [Generalized Pythagorean theorem]
		Let $\vartheta_P, \vartheta_Q, \vartheta_R$ be three points in the statistical manifold. Then
		\begin{equation} \label{eqn:pyth}
			{\bf L}_{\lambda, \varphi}[ \vartheta_Q : \vartheta_P] + {\bf L}_{\lambda, \varphi}[ \vartheta_R : \vartheta_Q] = {\bf L}_{\lambda, \varphi}[ \vartheta_R : \vartheta_P]
		\end{equation}
		if and only if the primal pre-geodesic from $\vartheta_Q$ to $\vartheta_R$ and the dual pre-geodesic from $\vartheta_Q$ to $\vartheta_P$ are $g$-orthogonal at $\vartheta_Q$. (When $\lambda > 0$, we assume that the divergences involved are all finite.)
	\end{theorem}
	\begin{proof}
		See \cite[Theorem 16]{W18}.
	\end{proof}
	
	Furthermore, by \cite[Theorem 15]{W18}, the geometric structure induced by a $\lambda$-logarithmic divergence has {\it constant} sectional curvature (in the information-geometric sense) equal to $\lambda$. (For a converse see \cite[Theorem 19]{W18}.) This justifies the interpretation of the constant $\lambda$ as the curvature. We refer the reader to \cite{WY19, WY19b} for some geometric interpretations of the curvature in terms of the logarithmic divergence between a primal-dual pair of geodesics. Some aspects of projections with respect to the $L^{(\alpha)}$-divergence (equivalent to a  $\lambda$-logarithmic divergence where $\alpha = -\lambda > 0$) are studied in \cite{MR21,TW21} where a dual foliation is constructed.

	\section{Duality between $\lambda$-exponential family and $\lambda$-logarithmic divergence}   \label{sec:dual.divergence}
	Our previous results show that the $\lambda$-exponential and $\lambda$-mixture families have a rich analytic and geometric structure when considered from the viewpoint of $\lambda$-duality and $\lambda$-logarithmic divergence. In this section we study some statistical implications of our approach. 
	
	\subsection{Motivation}
	An important property of exponential family which is perhaps less well known is that the log likelihood of an exponential family can be associated to a Bregman divergence. To give a concrete example, consider the normal location family $\{N(\theta, I)\}_{\theta \in \mathbb{R}^d}$. Let the reference measure $\nu$ be the standard normal distribution $N(0, I)$. Then the density of $N(\theta, I)$ is given by
	\[
	p(x; \theta) = e^{x \cdot \theta - \frac{1}{2} |\theta|^2},
	\]
	which is an exponential family with $F(x) = x$, $\phi(\theta) = \frac{1}{2}|\theta|^2$ and $\psi(\eta) = \phi^*(\eta) = \frac{1}{2}|\eta|^2$. We have $\eta = \nabla \phi(\theta) = \theta$, so the normal location family (with unit variance) is {\it self-dual}: the natural parameter $\theta$ is the same as the expectation parameter $\eta$. Using the identity $\frac{1}{2} |x - \eta|^2 = \frac{1}{2}|x|^2 - \theta \cdot x + \frac{1}{2}|\theta|^2$, we may write
	\begin{equation} \label{eqn:normal.log.likelihood}
		\log p(x; \theta) = -\frac{1}{2}|x - \eta|^2 + \frac{1}{2}|x|^2 = - {\bf B}_{\psi}[x : \eta] + \psi(x).
	\end{equation}
	Thus the log-likelihood can be expressed in terms of the Bregman divergence of $\psi$, which in this case is $1/2$ times the squared distance. This maybe regarded as a probabilistic basis of the least squares criterion. Next consider a general exponential family. By considering the distribution of $Y = F(X)$, we may consider exponential families on $\mathbb{R}^d$ of the form
	\[
	p(y ; \theta) = e^{\theta \cdot y - \phi(\theta)}.
	\]
	Let $\psi = \phi^*$ be the convex conjugate of $\phi$. In the seminal paper \cite{BMDG05}, it was proved  that when the family is absolutely continuous with respect to either the counting or Lebesgue measure, and satisfies natural regularity conditions, then one has the representation
	\begin{equation} \label{eqn:log.likliehood.divergence}
		p(y ; \theta) = e^{- {\bf B}_{\psi} [ y : \eta] + \psi(y)}.
	\end{equation}
	Note that in general it is the dual parameter $\eta = \nabla \phi(\theta) = \mathbb{E}_{\theta}[F(X)]$, not $\theta$, which has a direct interpretation on the ``data space" where $y$ lives. This representation implies that maximization of the likelihood is equivalent to minimization of the corresponding Bregman divergence. We also mention that
	\begin{equation} \label{eqn:Bregman.barycenter}
		\eta = \argmin_{\eta'} \mathbb{E}_{\theta} [ {\bf B}_{\psi}[ Y : \eta']], \quad Y \sim p(\cdot ; \theta).
	\end{equation}
	So the expectation parameter $\eta$ is also the right barycenter of the distribution with respect to the Bregman divergence ${\bf B}_{\psi}$.
	
	\subsection{$\lambda$-exponential family and $\lambda$-logarithmic divergence}
	In \cite[Section 4]{W18}, the $\mathcal{F}^{(\pm \alpha)}$-families were derived by replacing the Bregman divergence in \eqref{eqn:log.likliehood.divergence} by a logarithmic divergence. For example, the density \eqref{eqn:f.log.likelihood.as.c} of the Dirichlet perturbation model can be expressed in the form $p(\cdot) \propto e^{-{\bf D}}$, where ${\bf D}$, given by \eqref{eqn:Dirichlet.cost}, is a logarithmic divergence (see Example \ref{eg:excess.growth}). Here, it is natural to adopt divisive rather than subtractive normalization. In this subsection, we derive heuristically the analogue of \eqref{eqn:log.likliehood.divergence} for the $\lambda$-exponential family, and leave the rigorous treatment to future research.
	
	Following the approach of \cite{BMDG05}, consider a $\lambda$-exponential family on $\mathbb{R}^d$ where $F(y) = y$. The dominating measure is assumed to be absolutely continuous with respect to either the counting or Lebesgue measure. Under the support condition (so that we may assume $1 + \lambda \vartheta \cdot y > 0$ on a common domain independent of $\vartheta$), we may write the density as
	\begin{equation} \label{eqn:lambda.exp.likelihood.step1}
		\begin{split}
			p(y ; \vartheta) &= (1 + \lambda \vartheta \cdot y)^{1/\lambda} e^{-\varphi_{\lambda}(\vartheta)} \\
			&= e^{\frac{1}{\lambda} \log (1 + \lambda \vartheta \cdot y) - \varphi_{\lambda}(\vartheta)}.
		\end{split}
	\end{equation}
	Here $\varphi_{\lambda}(\vartheta)$ is the divisive $\lambda$-potential which is $c_{\lambda}$-convex. Let $\psi_{\lambda}$ be the $\lambda$-conjugate of $\varphi_{\lambda}$ which is also $c_{\lambda}$-convex. 
	
	Let $\eta = \nabla^{c_{\lambda}} \varphi_{\lambda}(\vartheta)$ be the dual parameter corresponding to $\vartheta$. 
	By the $c_{\lambda}$-duality, we have
	\begin{equation} \label{eqn:duality.identity1}
		\frac{1}{\lambda} \log (1 + \lambda \vartheta \cdot \eta) = \varphi_{\lambda}(\vartheta) + \psi_{\lambda}(\eta).
	\end{equation}
	Also, we have the identity (see \cite[(41)]{W18})
	\begin{equation} \label{eqn:duality.identity2}
		1 + \lambda \vartheta \cdot \eta = \frac{1}{1 - \lambda \nabla \psi_{\lambda}(\eta) \cdot \eta}.
	\end{equation}
	
	Substituting \eqref{eqn:duality.identity1} into \eqref{eqn:lambda.exp.likelihood.step1}, we see that the exponent can be written as
	\begin{equation} \label{eqn:lambda.exp.likelihood.step2}
		\frac{1}{\lambda} \log (1 + \lambda \vartheta \cdot y) - \frac{1}{\lambda} \log (1 + \lambda \vartheta \cdot \eta) + \psi_{\lambda}(\eta).
	\end{equation}
	Write $\vartheta = \nabla^{c_{\lambda}} \psi_{\lambda}(\eta) = \frac{\nabla \psi_{\lambda}(\eta)}{1 - \lambda \nabla \psi_{\lambda}(\eta) \cdot \eta}$. Using this and \eqref{eqn:duality.identity2}, we may rearrange \eqref{eqn:lambda.exp.likelihood.step2} to get
	\begin{equation*}
		\frac{1}{\lambda} \log (1 + \lambda \nabla \psi(\eta) \cdot (y - \eta)) + \psi_{\lambda}(\eta) 
		= - {\bf L}_{\lambda, \psi}[y : \eta] + \psi_{\lambda}(y),
	\end{equation*}
	where in the last equality it is implicitly assumed that $y$ belongs to the domain of $\psi_{\lambda}$ so that the divergence ${\bf L}_{\lambda, \psi}[y : \eta]$ is well-defined (this is the main subtlety; also see the proof of \cite[Theorem 4]{BMDG05}). Thus, we have shown heuristically that under suitable conditions, the density of the $\lambda$-exponential family can be expressed in the form
	\begin{equation} \label{eqn:log.likelihood.as.L.div}
		p(y ; \vartheta) = e^{- {\bf L}_{\lambda, \psi}[ y : \eta] + \psi_{\lambda}(y)}.
	\end{equation}
	
	\begin{remark}
		For an explicit $\lambda$-exponential family such as the Dirichlet perturbation model (Example \ref{eg:Dirichlet}), the representation \eqref{eqn:log.likelihood.as.L.div} can be verified directly (see \eqref{eqn:f.log.likelihood.as.c}). Note that \eqref{eqn:log.likelihood.as.L.div} does not hold for the Cauchy location-scale family (Example \ref{ex:Cauchy}); this is because the distribution of $Y = F(X) = (X, X^2)$ is supported on a parabola and hence is not absolutely continuous with respect to the Lebesgue measure on $\mathbb{R}^2$. The same issue also occurs in the case of exponential family.
	\end{remark}
	
	\begin{remark}
		Using results from classical convex analysis \cite{R70} and deeper mathematical properties of exponential families \cite{B14}, in particular convex functions of Legendre type, Banerjee et al.~\cite{BMDG05} also showed that there is a one-to-one correspondence between what they call regular Bregman divergences and regular exponential families. We expect analogous results hold for $\lambda$-logarithmic divergences and $\lambda$-exponential families.
	\end{remark}
	
	\subsection{Statistical consequences} \label{sec:statistical.conseq}
	The representation \eqref{eqn:log.likelihood.as.L.div} leads to some interesting consequences. In the following results we consider a $\lambda$-exponential family whose density satisfies \eqref{eqn:log.likelihood.as.L.div}.
	
	\begin{theorem} [MLE as right barycenter] \label{thm:MLE.as.barycenter}
		Let $\hat{\vartheta}$ be a maximum likelihood estimate with respect to the $\lambda$-exponential family $\{p(\cdot ; \vartheta)\}_{\vartheta \in \Omega}$ and  
		data points $y_1, \ldots, y_n$ under i.i.d.~sampling. Then $\hat{\eta} = \nabla^{c_{\lambda}} \varphi_{\lambda}(\hat{\vartheta})$ is a right barycenter of the data points with respect to ${\bf L}_{\lambda, \psi_{\lambda}}$:
		\[
		\hat{\eta} \in \argmin_{\eta \in \Omega'} \sum_{i = 1}^n \frac{1}{n} {\bf L}_{\lambda, \psi_{\lambda}} [ y_i : \eta].
		\]
	\end{theorem}
	\begin{proof}
		Consider observations $y_1, \ldots, y_n \in \mathcal{Y}$. By \eqref{eqn:log.likelihood.as.L.div}, the log-likelihood is given by
		\[
		\sum_{i = 1}^n \log p(y_i ; \vartheta) = \sum_{i = 1}^n \left( - {\bf L}_{\lambda, \psi_{\lambda}}[ y_i : \eta] + \psi_{\lambda}(y_i) \right).
		\]
		Since the second term is independent of $\vartheta$, maximizing the likelihood is equivalent to minimizing the sum of the divergences. It follows that the MLE $\hat{\eta}$ is the right barycenter.
	\end{proof}
	
	By Theorem \ref{thm:MLE.as.barycenter}, minimization of the $\lambda$-logarithmic divergence can be interpreted probabilistically as maximum likelihood estimation of an underlying $\lambda$-exponential family. This gives a probabilistic basis for using $\lambda$-logarithmic divergences as loss functions. Using the duality \eqref{eqn:log.likliehood.divergence} between exponential family and Bregman divergence, statistical methodologies such as clustering and principal component analysis were investigate in the literature \cite{BMDG05, CDS02}. Some preliminary results about these statistical methodologies using $\lambda$-logarithmic divergence are reported in the recent work \cite{TW21}.
	
	The next result, which is closely related to Theorem \ref{thm:MLE.as.barycenter}, provides a new probabilistic interpretation of the dual variable $\eta$. In classical deformation theory (Section \ref{sec:classic.deformation}), $\eta = \widetilde{\mathbb{E}}_{\vartheta}[Y]$ is the escort expectation which involves the escort distribution. Here, we show that the dual parameter has a probabilistic interpretation under the original distribution without performing the escort transformation.
	
	\begin{theorem} [Dual variable as right barycenter]\label{thm:dual.as.barycenter}
		The dual variable $\eta$ is the unique right barycenter of the distribution $p(\cdot ; \vartheta)$:
		\begin{equation} \label{eqn:dual.as.barycenter}
			\eta = \argmin_{\eta' \in \Omega'} \mathbb{E}_{\vartheta} \left[ {\bf L}_{\lambda, \psi_{\lambda}} [ Y : \eta'] \right].
		\end{equation}
	\end{theorem}
	\begin{proof}
		Let $\eta' \in \Omega'$ (corresponding to $\vartheta' \in \Omega$) be given. By \eqref{eqn:log.likelihood.as.L.div}, we have
		\begin{equation*}
			\begin{split}
				0 &\leq {\bf H}( p(\cdot ; \vartheta) || p(\cdot ; \vartheta')) \\
				&= \int p(y ; \vartheta) \log \frac{p(y ; \vartheta)}{p(y ; \vartheta')} d\nu(y) \\
				&= \mathbb{E}_{\vartheta} \left[ - {\bf L}_{\lambda, \psi_{\lambda}}[y : \eta] + {\bf L}_{\lambda, \psi_{\lambda}}[ y : \eta'] \right].
			\end{split}
		\end{equation*}
		Rearranging gives $\mathbb{E}_{\vartheta} [{\bf L}_{\lambda, \psi_{\lambda}}[ y : \eta]] \leq \mathbb{E}_{\vartheta} [{\bf L}_{\lambda, \psi_{\lambda}}[ y : \eta']]$ and the theorem is proved.
	\end{proof}
	
	%

	\section{Conclusion} \label{sec:conclusion}
	In this paper we showed that the $\lambda$-duality, which is a one-parameter deformation of convex duality, leads to fresh perspectives on the $q$-exponential family (where $q=1-\lambda$) and related concepts. In particular, the mathematical properties of the $\lambda$-exponential family under the $\lambda$-duality nicely parallel those of the exponential family. Furthermore, the $\lambda$-mixture family, which also conforms to the same duality, may be understood as another face of the $\lambda$-exponential family. In Table \ref{tab:persp} we summarize the analogy between the classical objects and the framework of this paper. 
	
	\begin{table}[h!]
		\centering
		\begin{tabular}{|p{5cm}|p{5cm}|} \hline
			{\bf Convex Duality} ($\lambda = 0$) & {\bf $\lambda$-Duality ($\lambda \neq 0$)}  \\ \hline
			convex function \eqref{eqn:convex.conjugate} & $c_{\lambda}$-convex function, where  \\& $\frac{1}{\lambda}(e^{\lambda f} - 1)$ is convex (Theorem \ref{thm:lambda.duality})\\ \hline
			Bregman divergence \eqref{eqn:Bregman.divergence} &  $\lambda$-logarithmic divergence (Definition \ref{def:L.div}) \\ \hline
			exponential family \eqref{eqn:exponential.family} & $\lambda$-exponential family (Definition \ref{def:lambda.exp.family}) \\ \hline
			mixture family \eqref{eqn:mixture.family} & $\lambda$-mixture family (Definition \ref{def:lambda.mixture}) \\ \hline
			Shannon entropy \eqref{eqn:Shannon.entropy} & R\'{e}nyi entropy \eqref{eqn:Renyi.entropy} \\ \hline
			KL-divergence \eqref{eqn:KL.divergence} & R\'{e}nyi divergence \eqref{eqn:Renyi.divergence} \\ \hline
			Dually flat geometry \cite{a16} & Dually projectively flat geometry (with constant  curvature $\lambda$) \cite{W18} \\ \hline
			$2$-Wasserstein transport \cite{V03,V08} & Dirichlet transport ($\lambda = -1$) \cite{PW16, PW18} (also see Example \ref{eg:Dirichlet})\\ \hline
		\end{tabular}
		\caption{Analogy between our framework and the classical one based on convex duality.}
		\label{tab:persp}
	\end{table}
	
	We expect that our framework will be helpful in further studies in the theory and application of generalized exponential families. A natural direction is to study implications of the $\lambda$-duality in the context of statistical physics. An explicit example is the {\it porous medium equation} which was analyzed in \cite{OW09} using the $q$-Gaussian distribution and information geometry (see also \cite{T12}). As described in Section \ref{sec:statistical.conseq}, the $\lambda$-logarithmic divergences leads to many potential applications in statistics and machine learning, some of which are being investigated by the authors. The geometric concepts studied in this paper, such as the escort transformation, are closely related to the Aitchison geometry of the probability simplex in compositional data analysis \cite{EA20}, and it is of interest to explore deeper links between compositional data analysis and information geometry. Finally, we believe that the $\lambda$-duality and logarithmic divergences can be applied to optimization in both convex and non-convex settings.

	
	%

	\section*{Appendix}
	\label{sec:App}
	\begin{proof}[Proof of Theorem \ref{thm:lambda.duality}(i)]
		We follow the approach of \cite[Section 3]{W18} and adapt the notations there to our setting. Consider the function $\frac{1}{\lambda} e^{\lambda f}$ which is convex on $\Omega$ by assumption (here the constant term in $F_{\lambda}$ is dropped for convenience). The tangent hyperplane based at $u' \in \Omega$ is given by
		\begin{equation*}
			\begin{split}
				u \mapsto &\frac{1}{\lambda} e^{\lambda f(u')} + e^{\lambda f(u')} \nabla f(u') \cdot (u - u') \\
				&= \frac{1}{\lambda} e^{\lambda f(u')} [ (1 - \lambda \nabla f(u') \cdot u') + \lambda \nabla f(u') \cdot u].
			\end{split}
		\end{equation*}
		By convex duality, $\frac{1}{\lambda} e^{\lambda f}$ is the maximum of the tangent hyperplanes where $u'$ varies over $\Omega$:
		\begin{equation*}
			\begin{split}
				&\quad \frac{1}{\lambda} e^{\lambda f(u)} 
				= \max_{u' \in \Omega}  \frac{1}{\lambda} e^{\lambda f(u')} [ (1 - \lambda \nabla f(u') \cdot u') + \nabla f(u') \cdot u]  \\
				&= \max_{u' \in \Omega} \frac{1}{\lambda} e^{\lambda f(u')} (1 - \lambda \nabla f(u') \cdot u') \left[1 + \frac{\lambda \nabla f(u') \cdot u}{1 - \lambda \nabla f(u') \cdot u'} \ \right].
			\end{split}
		\end{equation*}
		Note that in this step we use the assumption that $1 - \lambda \nabla f(u') \cdot u'$ is strictly positive. (A sufficient condition is that $0$ belongs to the closure of the domain. This construction also motivates the definition of the $\lambda$-deformed Legendre transformation $v' = \nabla^{c_{\lambda}} f(u')$.) Now take logarithm and rearrange. Here we consider the case $\lambda > 0$; the other case $\lambda < 0$ can be handled similarly. For $u \in \Omega$, we have
		\[
		f(u) = \max_{v' \in \Omega} \left\{ \frac{1}{\lambda} \log (1 + \lambda u \cdot v') - \widetilde{g}(v') \right\},
		\]
		where $v' = \frac{\nabla f(u')}{1 - \lambda \nabla f(u') \cdot u'}$ and $\widetilde{g}(v')$ absorbs the other terms which depend only on $v'$. 
		
		Since the mapping $u' \mapsto v'$ is a diffeomorphism from $\Omega$ to the range $\Omega'$ (see Proposition 1 and Theorem 11 of \cite{W18}), we may rewrite the above as
		\begin{equation*}
			\begin{split}
				f(u) &= \max_{v' \in \Omega'} \left\{ \frac{1}{\lambda} \log (1 + \lambda u \cdot v') - g(v') \right\} \\
				&= \max_{v' \in \Omega'} \left\{ -c_{\lambda}(u, v') - g(v') \right\},
			\end{split}
		\end{equation*}
		where $g(v') = \widetilde{g}(u')$. This shows that $f$ is $c_{\lambda}$-convex on $\Omega$. 
	\end{proof}
	

	\section*{Acknowledgment}
	The first author acknowledges support by NSERC Discovery Grant RGPIN-2019-04419 and a Connaught New Researcher Award. The second author acknowledges support by AFOSR Grant FA9550-19-1-0213. We also thank the referees and the associate editor for their careful reading and helpful comments.

	\bibliographystyle{abbrv}
	\bibliography{geometry.ref}
\end{document}